\documentclass[11pt,letterpaper, reqno]{amsart}

\usepackage[T1]{fontenc}	
\usepackage{amsmath, amsthm, amssymb}
\usepackage{mathrsfs}

\usepackage{cite}
\usepackage[dvipsnames]{xcolor}
\usepackage[inline,shortlabels]{enumitem}
\usepackage[hyphens]{url}									
\usepackage{bm}

\usepackage{mathtools}
\usepackage{xparse}
\usepackage{microtype}

\usepackage[a4paper,
  textwidth=17.0cm,
  textheight=24.0cm,
  ]{geometry}

\usepackage[pdftex,bookmarks,
						hidelinks,					
						colorlinks,				
						breaklinks,
						citecolor=OliveGreen,
						linkcolor=Maroon,
						]{hyperref} 

\allowdisplaybreaks

\ExplSyntaxOn
\NewDocumentCommand{\makeabbrev}{mmm}
 {
  \yoruk_makeabbrev:nnn { #1 } { #2 } { #3 }
 }

\cs_new_protected:Npn \yoruk_makeabbrev:nnn #1 #2 #3
 {
  \clist_map_inline:nn { #3 }
   {
    \cs_new_protected:cpn { #2 } { #1 { ##1 } }
   }
 }
 \ExplSyntaxOff
 
\makeabbrev{\textbf}{tbf#1}{a,b,c,d,e,f,g,h,i,j,k,l,m,n,o,p,q,r,s,t,u,v,w,x,y,z,A,B,C,D,E,F,G,H,I,J,K,L,M,N,O,P,Q,R,S,T,U,V,W,X,Y,Z}

\makeabbrev{\textbf}{bf#1}{a,b,c,d,e,f,g,h,i,j,k,l,m,n,o,p,q,r,s,t,u,v,w,x,y,z,A,B,C,D,E,F,G,H,I,J,K,L,M,N,O,P,Q,R,S,T,U,V,W,X,Y,Z}

\makeabbrev{\textsf}{tsf#1}{a,b,c,d,e,f,g,h,i,j,k,l,m,n,o,p,q,r,s,t,u,v,w,x,y,z,A,B,C,D,E,F,G,H,I,J,K,L,M,N,O,P,Q,R,S,T,U,V,W,X,Y,Z}

\makeabbrev{\mathsf}{mss#1}{a,b,c,d,e,f,g,h,i,j,k,l,m,n,o,p,q,r,s,t,u,v,w,x,y,z,A,B,C,D,E,F,G,H,I,J,K,L,M,N,O,P,Q,R,S,T,U,V,W,X,Y,Z}

\makeabbrev{\mathfrak}{mf#1}{a,b,c,d,e,f,g,h,i,j,k,l,m,n,o,p,q,r,s,t,u,v,w,x,y,z,A,B,C,D,E,F,G,H,I,J,K,L,M,N,O,P,Q,R,S,T,U,V,W,X,Y,Z,
sl,gl
}

\makeabbrev{\mathrm}{mrm#1}{a,b,c,d,e,f,g,h,i,j,k,l,m,n,o,p,q,r,s,t,u,v,w,x,y,z,A,B,C,D,E,F,G,H,I,J,K,L,M,N,O,P,Q,R,S,T,U,V,W,X,Y,Z}

\makeabbrev{\mathbf}{mbf#1}{a,b,c,d,e,f,g,h,i,j,k,l,m,n,o,p,q,r,s,t,u,v,w,x,y,z,A,B,C,D,E,F,G,H,I,J,K,L,M,N,O,P,Q,R,S,T,U,V,W,X,Y,Z}

\makeabbrev{\mathcal}{mc#1}{A,B,C,D,E,F,G,H,I,J,K,L,M,N,O,P,Q,R,S,T,U,V,W,X,Y,Z}

\makeabbrev{\mathbb}{mbb#1}{A,B,C,D,E,F,G,H,I,J,K,L,M,N,O,P,Q,R,S,T,U,V,W,X,Y,Z}

\makeabbrev{\mathscr}{ms#1}{A,B,C,D,E,F,G,H,I,J,K,L,M,N,O,P,Q,R,S,T,U,V,W,X,Y,Z}

\makeabbrev{\mathrm}{#1}{
id,ran,rk,diag,stab,ann,conv,pr,ev,tr,End,Hom,sgn,im,op,can,fin,ext,red,tot,lex,Aut,Inn,unit,hor,ver,new,eucl,old,
rot,usc,lsc,Lip,lip,bSymLip,osc,AC,loc,coz,z,
supp,Opt,Adm,Cpl,Geo,GeoOpt,GeoAdm,GeoCpl,reg,res,
bd,co,Ric,Exp,dExp,dist,seg,Seg,cut,fcut,Cut,SDiff,Iso,Isom,diam,cl,Homeo,Diff,Der,vol,dvol,inj,relint, Graph, sub,
var,law,Var,Poi,Gam,pa,so,iso,fs,inv,pqi,mix,erg,form,
TestF,
ob,cod,inp,
}

\makeabbrev{\mathsf}{#1}{CD,BE,MCP,Ent,wMTW,MTW,Ch,RCD,EVI,Rad,dRad,SL,cSL,dSL,ScL,Irr,SC,wFe,VA,MetMeas,UMeas,CSMet,Met,USp,Meas,Mbl,alg,Alg}

\makeabbrev{\mathsc}{#1}{mmaf,cg}

\renewcommand{\div}{\mathrm{div}}

\newcommand{\mathsc}[1]{\text{\textsc{#1}}}

\DeclareMathOperator{\eqdef}{\coloneqq}

\newcommand{\diff}{\mathop{}\!\mathrm{d}}						
		
\newcommand{\set}[1]{\left\{#1\right\}}							
\newcommand{\paren}[1]{\left(#1\right)}							
\newcommand{\braket}[1]{\left[#1\right]}

\newcommand{\scalprod}[2]{\left\langle #1 \,,\, #2\right\rangle}		
\newcommand{\Cb}{\mcC_b}									

\DeclareMathOperator{\car}{\mathbf 1}

\newcommand{\N}{{\mathbb N}}
\newcommand{\R}{{\mathbb R}}
\DeclareMathOperator{\Q}{{\mathbb Q}}

\newcommand{\restr}[1]{\big\lvert_{#1}}

\newcommand{\comma}{\,\mathrm{,}\;\,}
\newcommand{\semicolon}{\,\mathrm{;}\;\,}
\newcommand{\fstop}{\,\mathrm{.}}

\let\temp\phi
\let\phi\varphi
\let\varphi\temp

\let\temp\epsilon
\let\epsilon\varepsilon
\let\varepsilon\temp
\newcommand{\eps}{\epsilon}

\numberwithin{equation}{section}
\theoremstyle{plain}
\newtheorem{theorem}{Theorem}[section]

\newtheorem{proposition}[theorem]{Proposition}
\newtheorem{lemma}[theorem]{Lemma}
\newtheorem{corollary}[theorem]{Corollary}

\theoremstyle{definition}

\theoremstyle{remark}
\newtheorem{remark}[theorem]{Remark}
\newtheorem{example}[theorem]{Example}

\renewcommand{\paragraph}[1]{\medskip\emph{#1}.\quad}

\newcommand{\cic}{\mcC_{T}}
\newcommand{\acic}{\mathcal{AC}^2_{T}}

\begin{document}

\title[Superposition for the continuity equation with reaction]{A Lagrangian superposition principle for the continuity equation with reaction}

\author[S. Almi]{Stefano Almi}
\address[S. Almi]{Department of Mathematics and Applications ``R. Caccioppoli'', University of Naples
Federico II, Via Cintia, Monte S. Angelo, I-80126 Napoli, Italy}
\email{stefano.almi@unina.it} 

\author[G.~E. Sodini]{Giacomo Enrico Sodini}
\address[G.~E. Sodini]{Institute of Analysis and Scientific Computing, TU Wien,
  Wiedner Hauptstra\ss e 8-10, A-1040, Vienna, Austria}
\email{giacomo.sodini@tuwien.ac.at}

\begin{abstract}
We study the continuity equation with reaction~$\partial_t\mu + \div(v\mu) = w\mu$ on~$[0,T]\times\R^d$, driven by Borel velocity and reaction fields~$v$ and~$w$. We provide the first version of a Lagrangian superposition principle for the equation in its natural generality, assuming only finite quadratic energy without imposing boundedness conditions on the reaction term neither any form of compactness for the support of the measure. More precisely, every solution~$\mu \in \mcC([0,T];\msM_+(\R^d))$ with finite quadratic energy is represented by a probability measure~$\eta$ on absolutely continuous curves in the geometric cone over~$\R^d$, concentrated on the solutions of the characteristic system~$x' = v(x)$,~$k' = w(x)\,k$, through the $2$-homogeneous marginal~$h^2_t(\eta)=\mu_t$. We also establish the converse implication. The proof passes through a regularization of the triple~$(v,w,\mu)$ which produces fields satisfying only \emph{local} bounds; the core of the argument is therefore a representation theory under such local assumptions, whose key tool, of independent interest, is a two-time representation formula relating~$\mu_s$ and~$\mu_t$ along the flow of~$v$ for arbitrary times~$s,t$.
\end{abstract}


\maketitle

\setcounter{tocdepth}{2}
\tableofcontents

\section{Introduction}\label{s:Intro}
The goal of the present paper is to provide a novel superposition principle for solution to a {\em continuity equation with reaction}, removing unnecessary assumptions such as global boundedness of the reaction term~\cite{Man07} or {\em a priori} boundedness of the space domain~\cite{BreCarFanRom22}. Precisely, for a finite time horizon~$T>0$ and Borel maps~$v\colon[0,T]\times\R^d\to\R^d$ and~$w\colon[0,T]\times\R^d\to\R$, we deal with solutions $\mu\in\mcC([0,T];\msM_+(\R^d))$ to
\begin{equation}\label{eq:cerintro}
\partial_t\mu_t + \div(v_t\mu_t) = w_t\mu_t \qquad \text{in } \mcD'\big((0,T)\times\R^d\big) \comma
\end{equation}
where $\msM_+(\R^d)$ denotes the set of non-negative finite Borel measures on~$\R^{d}$ endowed with the narrow topology. In~\eqref{eq:cerintro} the velocity field~$v$ transports mass along its characteristics, while the scalar field~$w$ dilates it, increasing it where~$w>0$ and diminishing it where~$w<0$. Equation~\eqref{eq:cerintro} is the prototypical evolution underlying the dynamical formulation of {\em unbalanced optimal transport} \cite{Benamou2003}. A dynamic notion of distance between non-negative finite Borel measures on~$\R^{d}$ inspired by~\cite{BenBre} was indeed independently introduced and studied in~\cite{ChiPeySchVia2018a, KonMonVor2016}. A comprehensive study of unbalanced transport and the Hellinger--Kantorovich geometry is provided in the metric setting in the seminal papers~\cite{LieMieSav18, LieMieSav16} (see also~\cite{ChiPeySchVia2018b, SavSod24, DScSod25}). We further refer to~\cite{BreFan2020, LomMai, FitLauSch2017, Eyr2024} for some applications of unbalanced optimal transport to imaging, inverse problems, and machine learning.

\medskip

\noindent{\bf Superposition principles.} The Lagrangian, or probabilistic, representation of solutions to first-order evolution equations of the form~\eqref{eq:cerintro} has a long history. When the reaction term is absent, i.e.~$w\equiv0$ and~\eqref{eq:cerintro} reduces to the classical continuity equation~$\partial_t\mu + \div(v\mu)=0$, the superposition principle~\cite[Theorem~8.2.1]{AmbGigSav08} (cf.~also~\cite{Smirnov93}) represents any solution as a superposition of integral curves of~$v$, whenever $v$ is $p$-integrable with respect to the measure $\mu$, with $p \in (1, \infty)$. The theory has been extended to a purely metric-measure viewpoint in~\cite{AmbTre14, Tre16, Stepanov-Trevisan17, AmbRenVit26}, providing a link between the superposition principle and the decomposition of normal currents in simpler ones associated with rectifiable curves. In particular,~\cite{Stepanov-Trevisan17} extended the probabilistic representations previously obtained in~\cite{Lisini07, Lisini16}. We further mention recent superposition results obtained in the very general setting of Fokker--Planck--Kolmogorov equations in~\cite{BogRocSha21} and in the analysis of continuity equations with singular flux~\cite{Abedi-Zhenhao-Schultz24, AlmRosSav2025}, the latter being related to the study of bounded variation curves of probability measures.    

The key feature of the superposition principle is to establish a link between the Eulerian representation of the flow of measures solving the continuity equation and its Lagrangian 
description associated with the characteristic system of ODEs. Such a connection allows to extract finer information from the continuity equation, as proposed in \cite{Ambrosio04} for transport and continuity equations featuring velocitiy fields with low regularity (see also \cite{Ambrosio-Crippa08, Ambrosio-Bernard08, Ambrosio-Figalli09, AmbTre14}).

Let us mention that the superposition principle has been further employed for the well-posedness of mean-field particle systems in~\cite{Ambrosio-Fornasier-Morandotti-Savare21, Almi-DEramo-Morandotti-Solombrino, Bonafini-Fornasier-Schmitzer, Morandotti-Solombrino, Almi-Morandotti-Solombrino_JEE}, as well as for the analysis and finite-particle approximation of mean-field optimal control problems \cite{Albi-Almi-Morandotti-Solombrino,Almi-Durastanti-Solombrino1,Almi-Durastanti-Solombrino2,Almi-Morandotti-Solombrino_JDE,Cavagnari-Lisini-Orrieri-Savare22,Fornasier-Lisini-Orrieri-Savare19}, where, compared to previous literature (see, e.g.,~\cite{bongini2017optimal,Fornasier-Solombrino}), regularity restrictions on the control variable can be dropped.

In the presence of a reaction term, a representation on the \emph{geometric cone}~$\mfC[\R^{d}]$ over~$\R^d$, where the extra coordinate encodes the local mass and~$w$ governs its growth, was obtained by Maniglia~\cite{Man07} under a globally bounded reaction term~$w$ and, in the framework of the Hellinger--Kantorovich distance, by Bredies, Carioni, Fanzon and Romero~\cite{BreCarFanRom22} under the additional technical assumption that the support of $\mu$ is contained in a compact subset $\Omega$ of $\R^{d}$. The use of the geometric cone, instead of the simple product of~$\R$ with~$[0,\infty)$, is justified by the fact that, whenever the mass coordinate is equal to~$0$, the position coordinate is free, and does not need to solve an ODE, see~\eqref{eq:odegen} below. We also mention that, as a consequence of the Benamou-Brenier formulation of the Hellinger--Kantorovich distance in~\cite[Theorem 8.18]{LieMieSav18}, one can obtain a form of our superposition principle but only for the \emph{minimal} pair~$(v,w)$ driving the curve~$\mu$ and not for a general one.

\medskip

\noindent{\bf Our contribution.} The purpose of this work is to establish a superposition principle for~\eqref{eq:cerintro} on the whole space~$\R^d$ requiring \emph{no bound at all}, not even a local one, on the fields~$v$ and~$w$, which are only assumed to be Borel. Referring to Section~\ref{s:preliminaries} for the basic notation, the main result of our work reads as follows.

\begin{theorem}[Superposition principle]\label{thm:superconegen}
Let~$v\colon [0,T] \times \R^d \to \R^d$,~$w\colon [0,T] \times \R^d \to \R$ be Borel maps, and let~$\mu \in \mcC([0,T]; \msM_+(\R^d))$ be satisfying the following assumptions:
\begin{align}\label{eq:cerconegen}
&\vphantom{\int_{0}^{T}} \partial_t \mu + \div ( v \mu) = w\mu \quad \text{in } \mcD'((0,T) \times \R^d) \comma
\\ \label{eq:pbdconegen}
& \int_0^T \int_{\R^d} \braket{|v_t|^2 + |w_t|^2} \diff \mu_t \diff t < \infty \fstop
\end{align}
Then there exists a probability measure~$\eta \in \mcP(\mcC([0, T]; \mfC[\R^{d}]))$ such that:
\begin{enumerate}
\item $\eta$ is concentrated on curves~$\gamma = [x_{\gamma}, r_{\gamma}] \in \mathcal{AC}^{2}([0, T]; \mfC[\R^{d}])$ such that
\begin{equation}\label{eq:odegen}
 x_\gamma'(t) = v_t(x_\gamma(t)) \comma \qquad k_\gamma'(t) = w_t(x_\gamma(t))\, k_\gamma(t) \qquad \text{for a.e.~} t \in r_\gamma^{-1}\big( (0,\infty) \big)\,,
\end{equation}
where $k_{\gamma} = r_{\gamma}^{2}$;
\item $h_t^2(\eta) = \mu_t$ for every~$t \in [0,T]$, i.e., 
\begin{align*}
    \int_{\R^d} \phi \diff h_t^2(\eta) \eqdef \int_{\mcC([0, T]; \mfC[\R^{d}])} \phi(x_\gamma(t))\, r^2_\gamma(t) \diff \eta(\gamma) = \int_{\R^{d}} \phi \diff \mu_{t} \qquad \text{for every } \phi \in \Cb(\R^d)\,;
\end{align*}
\item it holds
\begin{equation}\label{eq:energyidgen}
 \int_{\mcC([0, T]; \mfC[\R^{d}])} \int_{0}^{T} |\gamma'|^{2} (t) \diff t  \diff\eta(\gamma) = \int_0^T \int_{\R^d} \braket{ |v_t|^2 + \tfrac14 |w_t|^2 } \diff\mu_t \diff t \fstop
\end{equation}
\end{enumerate}
\end{theorem}

We remark that the sole quantitative hypothesis in Theorem~\ref{thm:superconegen} is the finiteness of the quadratic energy of the solution~$\mu$ itself~\eqref{eq:pbdconegen}. The extension of the results~\cite{Man07, BreCarFanRom22} is not a mere technicality. As we discuss in Example~\ref{example:intro} below and in Lemma~\ref{le:global}, creation of mass forces~$w$ to be unbounded from above near the region where mass appears, so that no global bounds on~$w$ can be assumed if the equation is to be treated in its natural generality on~$\R^d$. Furthermore, creation and annihilation of mass are linked to finite-time blow-up of the first equation in~\eqref{eq:odegen}, which in turn implies that the support of~$\mu_{t}$ can not be restricted to a compact subset of~$\R^{d}$. This is a major difference with respect to~\cite{Man07, BreCarFanRom22}, where a priori bounds were exploited to ensure tightness and compactness in the classical approximation method of~\cite[Theorem~8.2.1]{AmbGigSav08}. In our setting the space component~$x_{\gamma}$ is not uniformly controlled, as it may blow-up in finite time, even for smooth fields~$v$ and~$w$. When this happens, however, $r_{\gamma}$ must tend (or be equal) to~$0$, so that $\gamma$ tends to the vertex of the cone~$\mfC[\R^{d}]$. Hence, we will see in Section~\ref{s:general} that compactness is restored working in~$\mfC[\R^{d}]$.

Before proceeding with the description of the proof of Theorem~\ref{thm:superconegen}, we present a simple example concerning the representation of a curve of measures $\mu \in \mcC([0, \pi]; \msM_{+} (\R))$ solution to the continuity equation with reaction~\eqref{eq:cerintro} with regular fields $v$ and $w$. The main feature of the following example is that the characteristics~\eqref{eq:odegen} suffer a blow-up in finite time, justifying our analysis.

\begin{example}
    \label{example:intro}
Let~$d=1$,~$T=\pi$ and set
\begin{align*}
\alpha(t) &\eqdef \sgn(\pi/2-t) \comma \quad && t \in [0,\pi] \comma
\\
v_t(x) &\eqdef \alpha(t)(1+x^2) \comma \quad w_t(x) \eqdef -\alpha(t)\frac{5}{\pi/2-\arctan(x)} \quad &&t \in [0,\pi] \comma x \in \R \comma
\\
m(t) & \eqdef\alpha(t)  \paren{\frac{\pi/2-t}{\pi/2}}^5\,, \qquad g(t) \eqdef \alpha(t) \tan (t) \qquad && t \in [0, \pi] \setminus \set{\pi/2}\,,\\
\mu_t &\eqdef m(t) \delta_{g(t)} \quad &&t \in [0,\pi] \fstop
\end{align*}
The maps~$v$ and $w$ are Borel,~$v_t \in \mcC^\infty(\R)$,~$w_t \in \mcC^\infty(\R)$ for every~$t \in [0,\pi]$. Since~$m$ and~$g$ are continuous in~$[0,\pi] \setminus \set{\pi/2}$ and, when~$t \to \pi/2^{\pm}$, even if~$g(t) \to \infty$, the weight~$m(t) \to 0$, we have that~$\mu \in \mcC([0,\pi]; \msM_+(\R))$. Furthermore, it is easy to see that 
\begin{align}
    \label{eq:w1locloc}
& \int_0^\pi \braket{ \sup_{B} |v_t| + \Lip(v_t, B) + \sup_{ B}|w_t| +\Lip(w_t, B) } \diff t < \infty \comma B \Subset \R \,.
\end{align}

We show that ($v,w,\mu)$ satisfies~\eqref{eq:cerintro} and~\eqref{eq:pbdconegen}. Noticing that~$\mu_t = \mu_{\pi-t}$,~$v_t=-v_{\pi-t}$,~$w_t= -w_{\pi-t}$ for every~$t \in [\pi/2,\pi]$, we can limit the rest of our analysis to the interval~$[0,\pi/2]$. The flow of~$v$ starting from~$0$ is
\[
X_t(0)= \tan(t) \comma \quad t \in [0, \pi/2)
\]
which is not globally defined. We show that~\eqref{eq:cerintro} in~$(0, \pi/2)$ is satisfied: let $\phi \in \mcC^\infty_c((0,\pi/2) \times \R)$ be fixed. We compute the derivative of the map
\[
[0,\pi/2) \ni t \mapsto \xi(t) \eqdef \phi(t, g(t)) m(t) = \phi(t,\tan(t))\paren{\frac{\pi/2-t}{\pi/2}}^5
\]
which is given by
\begin{align*}
\xi'(t) &=\paren{\partial_t \phi(t, \tan(t)) + \partial_x \phi(t, \tan(t))(1+\tan^2(t))} \paren{\frac{\pi/2-t}{\pi/2}}^5
\\
&\quad - \phi(t,\tan(t)) \frac{10}{\pi} \paren{\frac{\pi/2-t}{\pi/2}}^4
\\
& = \braket{ \partial_t \phi(t, \tan(t)) + \partial_x \phi(t, \tan(t))(1+\tan^2(t)) +\phi(t,\tan(t)) \frac{-5}{\pi/2-t} } \paren{\frac{\pi/2-t}{\pi/2}}^5
\\
& = \braket{ \partial_t \phi(t, \tan(t)) + \partial_x \phi(t, \tan(t))v_t(\tan(t)) +\phi(t,\tan(t)) w_t(\tan(t)) } \paren{\frac{\pi/2-t}{\pi/2}}^5
\\
& = \braket{ \partial_t \phi(t, g(t)) + \partial_x \phi(t, g(t))v_t(g(t)) +\phi(t,g(t)) w_t(g(t))} m(t) \fstop
\end{align*}
Integrating in~$[0,\pi/2]$ gives the desired result. We are only left to verify~\eqref{eq:pbdconegen} in~$(0, \pi/2)$: we compute
\begin{align*}
\int_0^{\pi/2} \int_{\R} \braket{ |v_t|^2 + |w_t|^2} \diff \mu_t \diff t = \int_0^{\pi/2} \braket{ \paren{1+\tan^2(t)}^2+ \paren{\frac{5}{\pi/2-t}}^2 } \paren{\frac{\pi/2-t}{\pi/2}}^5 \diff t \fstop
\end{align*}
Notice that the integrand function is continuous in~$[0,\pi/2)$ and it vanishes as~$t \uparrow \pi/2$. Therefore the integral is finite.

We remark that although the flow associated to~$v$ is not everywhere defined in $[0, \pi]$ and solutions starting at $x=0$ for $t=0$ blow up at $t = \pi/2$, the curve $\mu$ still satisfies the superposition principle Theorem~\ref{thm:superconegen}. In this example, mass gets destroyed and created at time~$t=\pi/2$ so that also the curves
\[
\nu_t^1 \eqdef \mu_t \car_{[0,\pi/2]}(t) + \bm{0} \car_{[\pi/2, \pi]}(t) \comma \quad  \nu^2_t \eqdef \bm{0} \car_{[0,\pi/2]}(t) + \mu_t \car_{[\pi/2, \pi]}(t) \comma \quad t \in [0,\pi]
\]
are solutions to the continuity equation with reaction driven by~$v$ and~$w$, where $\bm{0} $ denotes the zero measure in~$\R$. 
\end{example}

\medskip

\noindent{\bf Outline of the proof.} The proof of Theorem~\ref{thm:superconegen} follows the general steps of~\cite[Theorem~8.2.1]{AmbGigSav08} (see also~\cite{Man07, BreCarFanRom22}). Precisely,  we regularize the triple~$(v,w,\mu)$ by mollification and adding a small Gaussian regularization to keep the mollified measures strictly positive. This produces a smooth triple $(v^{\eps}, w^{\eps}, \mu^{\eps})$ solution to~\eqref{eq:cerintro} that satisfies only \emph{local} bounds (cf.~\eqref{eq:w1loccone}), inherited from the integrability~\eqref{eq:pbdconegen} of~$v,w$ against~$\mu$, and no global control on~$w^\eps$ survives. 

The crucial step is therefore to have at hand a representation theory valid under purely local assumptions on~$v^{\eps}$ and~$w^{\eps}$. The fundamental result, of independent interest, is a \emph{two-time representation formula} (see Lemma~\ref{le:twotime}) established under such local assumptions: denoting by~$X^{\eps}_{s\to t}$ the flow of~$v^{\eps}$, by
\[
W^{\eps}_{s\to t}(x)\eqdef\exp\bigg(\int_s^t w^{\eps}_r(X^{\eps}_{s\to r}(x))\diff r\bigg)
\]
the accumulated reaction weight, and by~$E^{\eps}_{s\to t}$ the set of points whose trajectory, alive at time~$s$, survives up to time~$t$, we prove that
\begin{equation}\label{eq:twotimeintro}
(X^{\eps}_{s\to t})_\sharp\big(W^{\eps}_{s\to t}\,\mu^{\eps}_s\restr{E^{\eps}_{s\to t}}\big)=\mu^{\eps}_t\restr{E^{\eps}_{t\to s}}\qquad\text{for every } s,t\in[0,T]\fstop
\end{equation}
The two-time formula~\eqref{eq:twotimeintro} is what makes it possible to work under local assumptions only, that is, without any global bound on~$w$. Indeed, when~$w^{\eps}$ is unbounded from above, as is unavoidable already at the level of the mollified fields, the characteristic flow of~$v^{\eps}$ need not be globally defined and mass may be created or annihilated at finite times (see Example~\ref{example:intro}). Consequently, neither uniqueness nor a representation \emph{anchored at the initial time}~$t=0$ can hold in general. Formula~\eqref{eq:twotimeintro} circumvents this obstruction by comparing~$\mu^{\eps}_s$ and~$\mu^{\eps}_t$ directly along the flow between any two times, on the sets~$E^{\eps}_{s\to t}$ where the trajectories are actually defined, without ever referring to a fixed reference time. The proof of Lemma~\ref{le:twotime} rests on a comparison principle for sub- and supersolutions of~\eqref{eq:cerintro}, which we establish in Proposition~\ref{prop:uniqu} and Corollary~\ref{cor:comp}. Then, in Theorem~\ref{thm:supercone} we prove an intermediate superposition principle under local Lipschitz assumptions~\eqref{eq:w1loccone} on~$v^{\eps}$ and~$w^{\eps}$, where the flow~$X_{s\to t}^\eps$ and the weight~$W_{s\to t}^\eps$ are well-defined and~\eqref{eq:twotimeintro} can be exploited directly. The lift of~$\mu^{\eps}$ to the cone is built by selecting, along each trajectory, a reference time at its ``life-midpoint'', which turns the two-time formula into the desired marginal identity (see Theorem~\ref{thm:superconegen}(2) and Theorem~\ref{thm:supercone}(2)). A tightness argument on the cone, based on the energy bound~\eqref{eq:pbdconegen}, yields a limit measure~$\eta$ that represents the original~$\mu$. The fact that~$\eta$ is concentrated on genuine solutions to the characteristic system is recovered from the weak formulation by a density argument in the test functions, thus completing the proof of Theorem~\ref{thm:superconegen}. We finally complement Theorem~\ref{thm:superconegen} with the converse implication in Proposition~\ref{prop:converse}: any~$\eta$ concentrated on solutions of the characteristic system produces, through~$h^2_t$, a solution of~\eqref{eq:cerintro}, so that~\eqref{eq:cerintro} and its Lagrangian formulation on the cone are fully equivalent.

\medskip

\noindent{\bf Perspectives.} We expect the two-time representation formula and the ensuing superposition principle to be useful in several directions: the analysis of gradient flows and evolution variational inequalities in the (spherical) Hellinger--Kantorovich space, uniqueness and stability of measure-valued solutions to~\eqref{eq:cerintro}, the study of dynamic inverse problems with unbalanced optimal-transport regularization~\cite{BreCarFanRom22,BreFan2020}, and the probabilistic representation of Markov processes with killing and branching associated with~\eqref{eq:cerintro}. The superposition principle of Theorem~\ref{thm:superconegen} may be further extended to more general state spaces, replacing for instance $\R^{d}$ with a separable Banach space. This may pave the way to the analysis of multi-agent multi-label systems with reaction, in the spirit of~\cite{Ambrosio-Fornasier-Morandotti-Savare21, Morandotti-Solombrino}, and to related control problems~\cite{Albi-Almi-Morandotti-Solombrino, Cavagnari-Lisini-Orrieri-Savare22, Fornasier-Lisini-Orrieri-Savare19}.

\medskip

\noindent{\bf Plan of the paper.} Section~\ref{s:preliminaries} collects the notation and the basic facts on curves and measures in the geometric cone. Section~\ref{s:comparison} establishes the comparison principle for~\eqref{eq:cerintro} under local, one-sided assumptions on~$w$. In Section~\ref{s:repr} we prove the two-time representation formula~\eqref{eq:twotimeintro} (see Lemma~\ref{le:twotime}) and the superposition principle under local Lipschitz bounds (Theorem~\ref{thm:supercone}). In Lemma~\ref{le:global} we further discuss how global bounds on~$w$ imply a representation anchored at the initial time. Finally, Section~\ref{s:general} is devoted to the proof of the general superposition principle Theorem~\ref{thm:superconegen} and to its converse Proposition~\ref{prop:converse}.

\section{Preliminaries and notation}
\label{s:preliminaries} 
In the whole paper,~$d \in \N_+$ is a positive natural number and~$T \in (0,\infty)$ is a finite time horizon. 

Given a measure space~$(\Omega, \mcH, \mssm)$ and a subset~$E$ of a Banach space~$(\mathbb B, \| \cdot \|_{\mathbb{B}})$, for~$p \in [1,\infty)$, we denote by~$L^p(\Omega, \mssm; E)$ the space of (equivalence classes of)~$p$-integrable~$E$-valued functions defined on~$\Omega$; when~$E=\R$, we omit it from the notation; when, in addition,~$\Omega \subset \R$, it is understood that the underlying measure is the one dimensional Lebesgue measure, and we simply write~$L^p(\Omega)$.

Given a complete and separable metric space~$(X,\mssd)$ and a metric space $(Y, \delta)$, we denote by~$\mcC(X; Y)$ (resp.~$\mcC_b(X; Y)$, $\mcC_c(X; Y)$) the space of continuous (resp.~continuous and bounded, continuous with compact support) $Y$-valued functions on~$X$. When~$Y=\R$, we omit it from the notation. If~$X\subset \R^d$, we add the superscript~$\infty$ to denote infinitely differentiable functions. The space $\mcC(X;Y)$ and its subsets are always endowed with the uniform distance induced by the distance $\delta$ on $Y$.

We denote by~$\msM(X)$ (resp.~$\msM_+(X)$) the space of real valued (resp.~non-negative and finite) Borel measures on~$X$, and by~$\mcP(X) \subset \msM_+(X)$ the subset of probability measures on~$X$. We endow~$\msM(X)$ and its subspaces with the narrow topology i.e.~the coarsest topology on~$\msM(X)$ which makes continuous the maps
\[
\msM(X) \ni \mu \mapsto \int_X \phi \diff \mu \comma \quad \phi \in \mcC_b(X) \fstop
\]
The total variation of a measure~$\mu \in \msM(X)$ is denoted by~$|\mu| \in \msM_+(X)$; if~$B \subset X$ is Borel, then~$\mu \restr{B}$ denotes the restriction of~$\mu$ to~$B$; the zero measure in~$X$ is denoted by~$\bm{0}$. For any interval $J \subseteq \R$, we will denote by $\mcC(J; \msM (X))$ (resp.~$\mcC(J; \msM_{+} (X))$) the set of continuous functions from $J$ to~$\msM (X)$ (resp.~$\msM_{+} (X)$) with respect to the narrow topology.

If~$f \colon X \to \R^m$,~$m\in \N_+$, and~$B \subset X$, then~$\Lip(f,B)$ denotes the Lipschitz constant of~$f$ on~$B$ i.e.
\[
\Lip(f,B) \eqdef \sup \set{ \frac{|f(x)-f(y)|}{\mssd(x,y)} : x, y \in B, \, x \ne y} \fstop
\]
When~$m=1$, the positive part of~$f$ is~$f^+ \eqdef f \vee 0$, and the negative part of~$f$ is~$f^- \eqdef -f \vee 0$, so that~$f=f^+-f^-$.

Recall that for an interval~$J \subset \R$, a continuous map~$\gamma \in \mcC(J; X)$ is called~$p$-absolutely continuous,~$p \in [1,\infty)$, if there exists~$g \in L^p(J)$ such that
\[
\mssd(\gamma_s, \gamma_t) \le \int_s^t g(r) \diff r \comma \quad s \le t, \, s,t \in J \fstop
\]
The set of such curves is denoted by~$\mcA\mcC^p(J; X)$ and we recall that it is a Borel subset of~$\mcC(J;X)$. The smallest such~$g$ is the metric derivative of~$\gamma$ and it is denoted by~$|\gamma'|$. We say that~$\gamma \in \mcA \mcC_{\loc}^p(J; X)$ if~$\gamma \in \mcA \mcC^p(I;X)$ for every compact subinterval~$I \subset J$.

For any set~$\Omega$, the function~$\car_\Omega$ is the characteristic function of~$\Omega$, equal to~$1$ on~$\Omega$ and~$0$ outside it. Finally, for~$R \in (0,\infty)$, we set
\[
B_R(0) \eqdef \set{ x \in \R^d : |x|<R} \fstop
\]

\subsection{Curves and measures in the geometric cone} 
\label{ssec:curvcone}

Recall that the \emph{geometric cone} over~$\R^d$ is the quotient
\[
\mfC[\R^d] \eqdef \big( \R^d \times [0,\infty) \big)\big/\big( \R^d \times \set{0} \big) \comma
\]
whose points are denoted~$[x,r]$, with~$x \in \R^d$ and~$r \ge 0$. All the points with~$r=0$ are identified with a single vertex~$\mfo$. We endow~$\mfC[\R^d]$ with the cone distance
\begin{equation}\label{eq:conedist}
\mssd_{\mfC}\big( [x_0,r_0], [x_1,r_1] \big)^2 \eqdef r_0^2 + r_1^2 - 2 r_0 r_1 \cos\big( |x_0-x_1| \wedge \pi \big) = |r_0-r_1|^2 + 4r_0 r_1 \sin^2 \paren{\frac{|x_0-x_1| \wedge \pi}{2}}\fstop
\end{equation}
For~$\lambda>0$ we further define~$\mfd_\lambda \colon \mfC[\R^d] \to \mfC[\R^d]$ as~$\mfd_\lambda([x,r]) \eqdef [x,\lambda r]$. 

We denote by $\cic \eqdef \mcC([0, T]; \mfC[\R^{d}])$ the space of continuous curves from $[0, T]$ with values in $\mfC[\R^{d}]$ and by~$\acic \eqdef\mathcal{AC}^2([0,T]; \mfC[\R^d])$ its subset of absolutely continuous curves with metric derivative in~$L^2([0,T])$. Every~$\gamma \in \cic$ is written as~$\gamma_t = [x_\gamma(t), r_\gamma(t)]$, where~$r_\gamma(t) \ge 0$ and~$x_\gamma(t) \in \R^d$ is uniquely determined whenever~$r_\gamma(t)>0$. We endow $\cic$ with the uniform distance induced by the cone distance~\eqref{eq:conedist}. 
For~$\gamma \in \cic$ we set
\begin{equation}\label{eq:notcurves}
O_\gamma \eqdef r_\gamma^{-1}\big( (0,\infty) \big) \comma \qquad \Sigma(\gamma) \eqdef \max_{t \in [0,T]} r_\gamma(t) \comma \qquad k_\gamma \eqdef r^2_\gamma \fstop
\end{equation}

We recall the following characterization of absolutely continuous curves in the cone, see~\cite[Lemma~8.1]{LieMieSav18}.

\begin{lemma}[Absolutely continuous curves in the cone]\label{le:accurves}
A curve~$\gamma \in \cic$ belongs to~$\acic$ if and only if~$x_\gamma \in \mathcal{AC}_{\loc}^2(O_\gamma; \R^d)$ and $r_\gamma \in \mathcal{AC}^2([0,T]; [0, \infty))$ with~$r_\gamma |x_\gamma'| \in L^2(O_\gamma)$. In this case, it holds
\begin{equation}\label{eq:metricderivative}
|\gamma'|^2(t) = \car_{O_\gamma}(t) \braket{| r_\gamma'|^2(t) + r^2_\gamma(t)\, |x_\gamma'|^2(t)} \qquad \text{for a.e.~} t \in [0,T] \fstop
\end{equation}
\end{lemma}
For~$t \in [0,T]$ we define the map~$h_t^2 \colon \msM_+(\cic) \to \msM_+(\R^d)$ by
\begin{equation}\label{eq:htdef}
\int_{\R^d} \phi \diff h_t^2(\eta) \eqdef \int_{\cic} \phi(x_\gamma(t))\, r^2_\gamma(t) \diff \eta(\gamma) \comma \quad \phi \in \Cb(\R^d) \fstop
\end{equation}
For any Borel map~$\lambda \colon \cic \to (0,\infty)$ we define $\mfD_\lambda \colon \cic \to \cic$ by
\[
\mfD_\lambda(\gamma) \eqdef \mfd_{\lambda(\gamma)} \circ \gamma\,.
\]
Finally, for every $\eta \in \msM_+\big(\cic\big)$ we set 
\begin{align}
    \label{e:Rlambda}
\mcR_\lambda \eta \eqdef (\mfD_\lambda)_\sharp \big( \lambda^{-2} \eta \big) \in \msM_+\big(\cic\big)\,.
\end{align}

The following lemma  is a simple consequence of the structure of dilations in the cone, see e.g.~\cite[eq.~(2.16), (2.17)]{DePSodTam25}.
\begin{lemma}[Renormalization]\label{le:dilation}
For every $\eta \in \msM_+\big(\cic\big)$, every Borel map $\lambda\colon \cic \to (0,\infty)$, and every~$t \in [0,T]$ we have~$h^2_t(\mcR_\lambda\eta) = h^2_t(\eta)$. Moreover, for every Borel~$\mcU \colon \cic \to [0,\infty]$ such that
\begin{align}
\label{e:homogeneity}
\mcU(\mfD_\lambda\gamma) = \lambda^2(\gamma)\mcU(\gamma)
\end{align}
we have~$ \int_{\mcC_{T}} \mcU \diff (\mcR_\lambda\eta) = \int_{\mcC_{T}}  \mcU \diff\eta$.
\end{lemma}

We now briefly discuss a compactness criterion in~$ \cic$. 
We define the~$2$-action functional $\mcF \colon \mcC_{T} \to [0, +\infty]$ as
\begin{equation}\label{eq:action}
\mcF(\gamma) \eqdef \begin{cases}
 \displaystyle \int_0^T |\gamma'|^2(t) \diff t \quad & \text{ if } \gamma \in \acic \comma \\[2mm]
 \infty \quad & \text{ else} \fstop
\end{cases}
\end{equation}
Given $\zeta \in \mcC(\R^{d}; [1, \infty))$ we define
\begin{align}
    \label{e:Gzeta}
\mcG_{\zeta} (\gamma) \eqdef \int_0^T k_\gamma(t)\, \zeta\big( x_\gamma(t) \big) \diff t \comma \quad  \gamma \in \cic \fstop
\end{align}

\begin{lemma}[Compactness in the cone]\label{le:conecompact}
Let~$\zeta \in \mcC(\R^{d}; [1, \infty))$. 
Then~$\mcF$ and~$\mcG_{\zeta}$ are lower semicontinuous. Moreover, if $\zeta$ has compact sublevel sets, then for every~$c>0$ the set
\[
\mcK_c \eqdef \set{ \gamma \in \cic : \Sigma(\gamma) \le 1, \ \mcF(\gamma) \le c, \ \mcG_{\zeta}(\gamma) \le c }
\]
is compact.
\end{lemma}
\begin{proof}
 Lower semicontinuity of~$\mcF$ is classical, see e.g.~\cite[Theorem 10.2]{AmbBruSem24}. As for~$\mcG_{\zeta}$, its lower semicontinuity follows by Fatou's lemma and the lower semicontinuity of the map~$\mfC[\R^d] \ni [x,r] \mapsto r^2 \zeta(x) \in [0,\infty)$.

Let us show that $\mcK_{c}$ is compact for $c>0$. Since~$\Sigma$ is continuous, $\mcK_c$ is closed and it suffices to prove that it is relatively compact. By the Cauchy--Schwarz inequality,
\begin{equation}\label{eq:holdercone}
\mssd_\mfC(\gamma_t,\gamma_s) \le \int_s^t |\gamma'|(\tau) \diff\tau \le \sqrt{c}\, |t-s|^{1/2} \comma \qquad s,t \in [0,T] \comma \, \gamma \in \mcK_c \comma
\end{equation}
so that~$\mcK_c$ is uniformly equicontinuous. By the Ascoli--Arzel\`a theorem it is then enough to check that, for every~$t\in[0,T]$, the set~$\set{\gamma_t : \gamma \in \mcK_c}$ is relatively compact in~$\mfC[\R^d]$. We recall that, by~\cite[Lemma 7.1]{LieMieSav18}, a subset~$A \subset \mfC[\R^d]$ is relatively compact as soon as 
\begin{equation}\label{eq:conerelcpt}
\sup\set{ r : [x,r] \in A } < \infty \quad \text{and} \quad \set{ x : [x,r] \in A, \ r \ge \delta } \text{ is bounded for every } \delta>0 \fstop
\end{equation}
 The first condition for~$\set{\gamma_t : \gamma \in \mcK_c}$ holds because~$\Sigma(\gamma)\le1$. As for the second one, we fix~$\delta \in (0,1]$,~$t \in [0,T]$ and we show that there exists a constant~$K>0$ (depending on~$\zeta, T, c, \delta$) such that, whenever~$\gamma \in \mcK_c$ is such that~$r_\gamma(t) \ge \delta$, then~$|x_\gamma(t)| \le K$. Let such a~$\gamma$ be given; by~\eqref{eq:holdercone} and~\eqref{eq:conedist} we have~$r_\gamma \ge \delta/2$ on~$J \eqdef [t-\varsigma,t+\varsigma] \cap [0,T]$, where~$\varsigma \eqdef \delta^2/(4c)$; note that~$|J| \ge \ell \eqdef \min\set{\varsigma, T}>0$. Since~$k_\gamma \ge \delta^2/4$ on~$J$, we get
\[
\int_J \zeta\big(x_\gamma(s)\big) \diff s \le \frac{4}{\delta^2}\, \mcG_{\zeta}(\gamma) \le \frac{4c}{\delta^2} \comma \qquad \int_J |x_\gamma'|^2(s) \diff s \le \frac{4}{\delta^2} \int_J r^2_\gamma(s) |x_\gamma'|^2(s) \diff s \le \frac{4c}{\delta^2} \comma
\]
where we used the definitions of~$\mcF$ and~$\mcG_{\zeta}$. By the first bound there is~$s_0 \in J$ with~$\zeta(x_\gamma(s_0)) \le 4c/(\delta^2 \ell)$, i.e.~$x_\gamma(s_0) \in K_\delta \eqdef \set{ \zeta \le 4c/(\delta^2\ell) }$, which is compact, hence contained in some ball of radius~$R>0$; note that~$R$ depends only on~$\zeta, T, c, \delta$. By the second bound, the fact that~$J \subset O_\gamma$, and the Cauchy--Schwarz inequality,
\[
\big| x_\gamma(t) - x_\gamma(s_0) \big| \le \int_J |x_\gamma'| \diff s \le |J|^{1/2} \paren{ \frac{4c}{\delta^2} }^{1/2} \le \frac{2\sqrt{cT}}{\delta} \fstop
\]
Hence~$|x_\gamma(t)| \le K\eqdef R+ \frac{2\sqrt{cT}}{\delta}$, which is the required bound.
\end{proof}

\section{The comparison principle for the continuity equation with reaction}\label{s:comparison}

The goal of this section is to prove the following comparison principle for the continuity equation with reaction. We remark that the comparison principle contained in Proposition~\ref{prop:uniqu} is a slightly more general version than the one in~\cite[Lemma 3.5]{Man07}, since we only require a global upper bound for~$w$. 

\begin{proposition}[Comparison for the continuity equation with reaction]
\label{prop:uniqu}
Let~$v \colon [0,T] \times \R^d \to \R^d$,~$w \colon [0,T] \times \R^d \to \R$ be Borel maps, and let~$\sigma \in \mcC([0,T]; \msM(\R^d))$ be satisfying the following assumptions:
\begin{align} \label{eq:w1loc}
&\int_0^T \braket{ \sup_{B} |v_t| + \Lip(v_t, B) + \sup_{B}w_t^- + \sup_{\R^d} w_t^+ + \Lip(w_t, B) } \diff t < \infty \comma B \Subset \R^d \comma
\\ \label{eq:cer}
& \vphantom{\int_{0}^{T}} \partial_t \sigma + \div ( v \sigma) \le w\sigma \text{ in } \mcD'((0,T) \times \R^d) \comma 
\\ \label{eq:pbd}
& \int_0^T \int_{\R^d} \braket{|v_t| + |w_t|} \diff |\sigma_t| \diff t < \infty \fstop
\end{align}
Moreover, assume that
\begin{align}
    \label{e:extra-assumption}
    \sigma_0 \le 0 \qquad \text{and} \qquad \sup_{t \in[0,T]} |\sigma_t|(B) < \infty \text{ for every~$B \Subset \R^d$.}
\end{align}
Then $\sigma_t \le 0$ for every $t \in [0,T]$.
\end{proposition}

Before proving Proposition~\ref{prop:uniqu}, we need a standard approximation lemma.

\begin{lemma}[A technical lemma]\label{le:techn} Let~$u \colon [0,T] \times \R^d \to \R^m$,~$m \in \N_+$, be a Borel map such that 
\begin{equation}
\label{eq:uw1loc}
\int_0^T \braket{ \sup_{B} |u_t| + \Lip(u_t, B) } \diff t < \infty \comma B \Subset \R^d \comma
\end{equation}
let~$R>0$, and let~$\sigma \in \mcC([0,T]; \msM(\R^d))$ be such that
\begin{align}
    \label{e:hp-sigma}
    \sup_{t \in [0,T]}|\sigma_t|(B_R(0)) < \infty\,.
\end{align} 
Then there exists a sequence~$(u^n)_{n} \subset \mcC^\infty(\R \times \R^d; \R^m)$ such that 
\begin{align}\label{eq:unbd}
\sup_{n \in \N} \int_0^T \sup_{\R^d} |u^n_t|\diff t &\le  \int_0^T  \sup_{\overline{B_{2R}(0)}} |u_t| \diff t \comma
\\ \label{eq:unw1}
\sup_{n \in \N} \int_0^T \braket{ \Lip(u^n_t, \R^d) } \diff t &\le \int_0^T \braket{  \Lip(u_t, \overline{B_{2R}(0)}) + \frac{2}{R}\sup_{\overline{B_{2R}(0)}} |u_t|  } \diff t
\end{align}
and
\begin{equation}\label{eq:uconv}
    \lim_{n \to \infty} \int_0^T \int_{B_R(0)} |u_t(x) - u^n_t(x)| \diff |\sigma_t|(x) \diff t =0 \fstop
\end{equation}
If~$m=1$, we also have
\begin{equation}\label{eq:globalupper}
\int_0^T \sup_{\R^d} (u^n_t)^+ \diff t \le \int_0^T \sup_{\overline{B_{2R}(0)}}  u_t^+ \diff t \fstop
\end{equation}
\end{lemma}
\begin{proof}
   We trivially extend~$u$ to~$\R \times \R^d$ and let~$\xi \in \mcC^\infty_c(\R^d)$ be such that~$\xi \equiv1$ on~$B_R(0)$,~$0 \le \xi \le 1$,~$\supp (\xi) \subset B_{2R}(0)$,~$\Lip(\xi, \R^d) \le 2/R$. Define~$k\eqdef u \xi$ and observe that~$k=u$ on~$[0,T] \times B_R(0)$ and
    \begin{equation}\label{eq:rbound}
    \sup_{\R^d} |k_t| \le \sup_{\overline{B_{2R}(0)} }|u_t|, \quad \Lip (k_t, \R^d) \le \Lip(u_t, \overline{B_{2R}(0)}) + \frac{2}{R}\sup_{\overline{B_{2R}(0)} }|u_t|  
    \end{equation}
    for every~$t \in [0,T]$. Therefore, using also~\eqref{eq:uw1loc}, we have
    \begin{equation}\label{eq:boundw}
     \sup_{\R^d} |k_\cdot| + \Lip (k_\cdot, \R^d) \in L^1([0,T]) \fstop
    \end{equation}
    Now we consider standard mollifiers~$(\varrho_\eps)_{\eps>0} \subset \mcC_c^\infty(\R^d)$ and~$(\tau_\eps)_{\eps>0} \subset \mcC_c^\infty(\R)$ and we set
    \[ u^{\eps}_t(x) \eqdef \int_{\R} \int_{\R^d} k_{t-s}(x-y) \varrho_\eps(y) \tau_\eps(s) \diff y \diff s\comma \quad (t,x) \in \R \times \R^d.\]
    Observe that~$u^{\eps} \in \mcC^\infty(\R \times \R^d; \R^m)$. Moreover, for every~$t \in \R$ and~$\eps >0$, we have
    \begin{align} \label{eq:thef}
    \sup_{\R^d} |u_t^{\eps}| &\le (\sup_{\R^d}|k_\cdot | \ast \tau_\eps)(t) \comma
    \\ 
    \label{eq:thef1}
    \Lip(u_t^\eps, \R^d) &\le (\Lip(k_\cdot, \R^d) \ast \tau_\eps)(t) \fstop
    \end{align}
    Since by~\eqref{eq:boundw} we have
    \begin{align*}
        \sup_{\R^d}|k_\cdot | \ast \tau_\eps \to \sup_{\R^d}|k_\cdot| &\quad \text{ in } L^1([0,T]) \text{ as } \eps \downarrow 0\comma
        \\
\Lip(k_\cdot, \R^d) \ast \tau_\eps \to \Lip(k_\cdot, \R^d) &\quad \text{ in } L^1([0,T]) \text{ as } \eps \downarrow 0 \comma
    \end{align*}
    the converse of Lebesgue dominated convergence theorem gives the existence of a subsequence~$(\eps_n)_n$ with~$\eps_n \downarrow 0$ as~$n \to  \infty$, a full measure subset $C \subset [0,T]$, and a function $h \in L^1([0,T])$ such that
    \begin{equation}\label{eq:hbound}
        (\sup_{\R^d}|k_\cdot | \ast \tau_{\eps_n})(t) + (\Lip(k_\cdot, \R^d) \ast \tau_{\eps_n})(t) \le h(t) \quad \text{ for every } n \in \N , \, t \in C.
    \end{equation}
    We claim that there exists a full measure subset~$A \subset [0,T]$ such that 
    \begin{equation}\label{eq:claim}
    u^{\eps_n}_t(x) \to k_t(x) \text{ for every } (t,x) \in  A \times \R^d \text{ as } n \to \infty \fstop
    \end{equation}
    Since  $k \in L^1_{\loc} ( [0,T]\times \R^d)$ by~\eqref{eq:boundw}, then~$u^{\eps_n} \to k$ in $ L^1_{\loc} ([0,T] \times \R^d)$ as~$n \to \infty$. Therefore, up to an unrelabeled subsequence, we can find a dense and countable subset~$J \subset \R^d$ such that for every~$x \in J$ the sets~$A_x \eqdef \{ t \in [0,T] : u^{\eps_n}_t(x) \to k_t(x) \text{ as } n \to \infty\}$ have full measure in~$[0,T]$. Define 
    \[ A\eqdef C \cap \bigcap_{x \in J } A_x \cap \set{t \in [0,T] : \Lip (k_t, \R^d) < \infty} \fstop \]
    Observe that~$A \subset [0,T]$ has full measure. Let~$(t,x) \in  A\times \R^d$ and let~$\delta>0$ be fixed. Since~$ J$ is dense in~$\R^d$, we can find~$y \in J$ such that $|x-y| < \delta$. By~\eqref{eq:thef1} and~\eqref{eq:hbound} we estimate
    \begin{align*}
        | u^{\eps_n}_t(x) - k_t(x) | &\le |u^{\eps_n}_t(x)-u^{\eps_n}_t(y)| + |u^{\eps_n}_t(y) - k_t(y)| + |k_t(y)-k_t(x)| 
        \\
        & \le \delta \left ( \Lip (u_t^{\eps_n}, \R^d) + \Lip (k_t, \R^d) \right ) +  |u^{\eps_n}_t(y) - k_t(y)|
        \\
        & \le \delta \left ( h(t)+ \Lip (k_t, \R^d) \right ) +  |u^{\eps_n}_t(y) - k_t (y)| \,.
    \end{align*}
     Since~$t \in A_y$, passing to the limit as~$n \to \infty$ and then as~$\delta \downarrow 0$ proves the claim in~\eqref{eq:claim}. Setting $u^n\eqdef u^{\eps_n}$, we have that clearly~\eqref{eq:unbd} and~\eqref{eq:unw1} hold because of~\eqref{eq:rbound},~\eqref{eq:thef},~\eqref{eq:thef1}, and the fact that mollification lowers the~$L^1$ norm. To prove~\eqref{eq:uconv} we apply the dominated convergence theorem with the measure
    \[ 
    |\sigma| \eqdef \int_0^T |\sigma_t |\restr{B_R(0)} \diff t \in \msM_+( [0,T] \times B_R(0)) \fstop 
    \]
    By~\eqref{eq:claim} and the fact that~$k$ coincides with~$u$ on~$B_R(0)$ we have that~$|u_t(x)-u_t^n(x)| \to 0$ as~$n \to \infty$ for~$|\sigma|$-a.e.~$(t,x) \in [0,T] \times B_R(0)$. We observe that by~\eqref{eq:thef} and~\eqref{eq:hbound} it holds that for~$|\sigma|$-a.e.~$(t,x) \in [0,T]\times B_R(0)$
    \begin{align*}
        |u_t(x)-u_t^n(x)| \le \sup_{\overline{B_R(0)}}|u_t| + \sup_{\R^d} |u_t^n| \le \sup_{\overline{B_R(0)}}|u_t| + h(t)\,.
    \end{align*}
     Since the right-hand side (as a function of~$(t,x)$) belongs to~$L^1( [0,T] \times B_R(0), |\sigma|)$ by~\eqref{eq:uw1loc} and~\eqref{e:hp-sigma} holds,  we can apply the dominated convergence theorem and conclude. The last assertion follows observing that
    \[
    k_t(x) \le  \sup_{\overline{B_{2R}(0)}} u_t^+ \comma \quad (t,x) \in [0,T] \times \R^d\comma
    \]
    the definition of~$u_t^\eps$, and the fact that mollification lowers the~$L^1$ norm.
\end{proof}

We now proceed with the proof of Proposition~\ref{prop:uniqu}.

\begin{proof}[Proof of Proposition~\ref{prop:uniqu}]
Fix~$\psi \in \mcC^{\infty}_c((0,T) \times \R^d)$~with $0 \le \psi \le 1$,~and a family of smooth cutoff functions~$(\chi_R)_{R>0} \subset \mcC_c^\infty(\R^d)$~such that~$\supp( \chi_R) \subset B_ {R}(0)$,~$\chi_R \equiv 1$ on~$B_{R/2}(0)$,~$0 \le \chi_R \le 1$ and~$\sup_{\R^d} |\nabla \chi_R | \le 4/R$. Let~$(v^{n, R})_n \subset \mcC^\infty(\R \times \R^d; \R^d)$ and~$(w^{n, R})_n \subset \mcC^\infty(\R \times \R^d)$  be the sequences of functions as in Lemma~\ref{le:techn} for~$u=v, m=d, R, \sigma$ and~$u=w, m=1, R, \sigma$, respectively.
We set 
\[ 
\phi^{n,R}_t(x) \eqdef - \int_t^T \psi_s(X^{n,R}_{t \to s}(x))e^{\int_t^s w_\tau^{n,R}(X^{n,R}_{t \to \tau}(x)) \diff \tau} \diff s, \quad (t,x) \in [0,T] \times \R^d\comma
\]
where~$X^{n,R}$ is the flow associated to~$v^{n,R}$, i.e.~for every~$(t,x) \in [0,T] \times \R^d$ the map~$[0,T] \ni s \mapsto X^{n,R}_{t \to s}(x) \in \R^d$ is smooth and satisfies
\[
X^{n,R}_{t \to t}(x)=x\comma \quad \tfrac{\diff}{\diff s} X^{n,R}_{t \to s}(x) = v_s^{n,R}(X^{n,R}_{t \to s}(x)) \text{ a.e.~} s \in (0,T)\fstop
\]
Note that the flow exists and is smooth because of the smoothness and the global bounds on~$v^{n,R}$, see e.g.~\cite[Lemma 8.1.4]{AmbGigSav08}. It is easy to check that~$\{\phi^{n,R}\}_{n,R} \subset \mcC^\infty([0,T]\times \R^d)$ and it holds
\[ 
\partial_t \phi^{n,R} + \langle v^{n,R}, \nabla \phi^{n,R} \rangle = \psi - w^{n,R}\phi^{n,R} \text{ in } (0,T) \times \R^d, \quad \phi^{n,R}_T=0 \text{ in } \R^d \fstop
\]
Since~$\phi^{n,R} \chi_R \in \mcC^\infty_c([0,T] \times \R^d)$ and $\phi^{n,R} \chi_R\leq 0$ in $[0, T ] \times \R^{d}$, it can be used to test~\eqref{eq:cer} getting
\begin{align} \label{eq:thefirstinthechain}
0 &\ge -\int_{\R^d}\phi^{n,R}_0 \chi_R \diff \sigma_0 
\\
&\ge \int_0^T \int_{\R^d} \braket{\chi_R \partial_t\phi^{n,R}_t + \scalprod{v_t}{ \chi_R \nabla\phi^{n,R}_t +\phi^{n,R}_t \nabla \chi_R} + w_t \phi^{n,R}_t \chi_R }\diff \sigma_t \diff t \nonumber
\\
&= \int_0^T \int_{\R^d} \chi_R \braket{ \psi_t + \scalprod{v_t - v_t^{n,R}}{\nabla\phi^{n,R}_t}+ (w_t-w_t^{n,R})\phi^{n,R}_t  }\diff \sigma_t \diff t \nonumber
\\
& \quad + \int_0^T \int_{\R^d}\phi^{n,R}_t \scalprod{\nabla \chi_R}{v_t} \diff \sigma_t \diff t \nonumber
\\
& \ge \int_0^T \int_{\R^d} \chi_R \psi_t \diff \sigma_t \diff t \nonumber
\\
& \quad - (\|\nabla \phi^{n,R}\|_\infty + \|\phi^{n,R}\|_\infty) \int_0^T \int_{B_R(0)} \braket{|v_t-v^{n,R}_t| + |w_t - w_t^{n,R}|} \diff |\sigma_t| \diff t\nonumber
\\ 
& \quad - \frac{4}{R} \|\phi^{n,R}\|_\infty \int_0^T \int_{B_R(0) \setminus B_{R/2}(0)} |v_t| \diff |\sigma_t| \diff t \fstop \nonumber
\end{align}

We proceed to estimate~$\| \phi^{n,R}\|_{\infty}$ and~$\| \nabla \phi^{n, R}\|_{\infty}$. For every~$(t,x) \in [0,T] \times \R^d$, using~\eqref{eq:globalupper} for $w^{n, R}$ we have
\begin{align}
\label{e:estimate-varphinR}
|\phi^{n,R}_t(x)| &\le \int_t^T |\psi_s(X^{n,R}_{t \to s}(x)) | e^{\int_t^s (w_\tau^{n,R}(X^{n,R}_{t \to \tau}(x)))^+ \diff \tau} \diff s \le \int_0^T e^{\int_0^T (w_\tau^{n,R}(X^{n,R}_{t \to \tau}(x)))^+ \diff \tau} \diff s
\\
& \le \vphantom{\int_{0}^{T}} T e^{\int_0^T \sup_{\R^d} (w_\tau^{n,R})^+ \diff \tau }  \le Te^{\int_0^T  \sup_{\overline{B_{2R}(0)}} w_t^+ \diff t }  \le Te^{\int_0^T  \sup_{\R^d} w_t^+ \diff t } \,. \nonumber
\end{align}
Regarding the gradient, for every~$(t,x) \in [0,T] \times \R^d$, we have by~\eqref{eq:unw1} that
\begin{align*}
| \nabla \phi^{n,R}_t(x) | & \le \int_t^T \bigg | \nabla \psi_s(X^{n,R}_{t \to s}(x)) D_x X^{n,R}_{t \to s}(x) 
\\
& \quad \quad +\psi_s(X^{n,R}_{t \to s}(x))\int_t^s \braket{ \nabla w_\tau^{n,R}(X^{n,R}_{t \to \tau}(x)) D_x X^{n,R}_{t \to \tau}(x) }\diff \tau  \bigg |
\\
& \quad\quad \cdot  e^{\int_t^s w_\tau^{n,R}(X^{n,R}_{t \to \tau}(x)) \diff \tau} \diff s
\\
& \le   \sup_{s,t \in [0,T]} \Lip(X^{n,R}_{t \to s}, \R^d) e^{\int_0^T  \sup_{\overline{B_{2R}(0)}} |w_t| \diff t }    \int_t^T \braket{\|\nabla \psi\|_\infty  + \int_t^s \Lip(w^{n,R}_\tau, \R^d) \diff \tau} \diff s
\\
& \le  T \sup_{s,t \in [0,T]} \Lip(X^{n,R}_{t \to s}, \R^d) e^{\int_0^T  \sup_{\overline{B_{2R}(0)}} |w_t| \diff t } 
\\
& \quad\quad \cdot \braket{ \|\nabla \psi\|_\infty + \int_0^T \braket{  \Lip(w_t, \overline{B_{2R}(0)}) + \frac{2}{R}\sup_{\overline{B_{2R}(0)}} |w_t|  } \diff t   } \,.
\end{align*}
By e.g.~\cite[Lemma 8.1.4]{AmbGigSav08}, we  have
\[
\sup_{s,t \in [0,T]} \Lip(X^{n,R}_{t \to s}, \R^d)  \le e^{S^{n,R}} \comma
\]
where
\begin{align*}
S^{n,R} &\eqdef \int_0^T (\sup_{\R^d} |v_t^{n,R}| + \Lip(v_t^{n,R}, \R^d)) \diff t \le \int_0^T \braket{  \paren{1+\tfrac{2}{R}} \sup_{ \overline{B_{2R}(0)}} |v_t| + \Lip(v_t, \overline{B_{2R}(0)})  } \diff t  \comma
\end{align*}
again by~\eqref{eq:unw1}. Overall, for every~$(t,x) \in [0,T] \times \R^d$, we have
\begin{align*}
| \nabla \phi^{n,R}_t(x) | &\le T  \exp \braket{ \int_0^T \braket{ \sup_{\overline{B_{2R}(0)}} |w_t| +  \paren{1+\tfrac{2}{R}} \sup_{ \overline{B_{2R}(0)}} |v_t| + \Lip(v_t, \overline{B_{2R}(0)})  } \diff t }
\\
& \quad \cdot \braket{ \|\nabla \psi\|_\infty + \int_0^T \braket{  \Lip(w_t, \overline{B_{2R}(0)}) + \frac{2}{R}\sup_{\overline{B_{2R}(0)}} |w_t|  } \diff t   } \fstop
\end{align*}
We deduce that
\[
C \eqdef \sup_{n,R} \|\phi^{n,R}\|_\infty < \infty\comma \quad C_R \eqdef \sup_{n} \|\nabla \phi^{n,R}\|_\infty < \infty \comma \quad R>0 \fstop
\]
Inserting these estimates in~\eqref{eq:thefirstinthechain}, we obtain
\begin{align*}
\int_0^T \int_{\R^d} \chi_R \psi_t \diff \sigma_t \diff t &\le (C_R+C) \int_0^T \int_{B_R(0)} \braket{|v_t-v^{n,R}_t| + |w_t - w_t^{n,R}|} \diff |\sigma_t| \diff t \\
&\quad +  \frac{4C}{R} \int_0^T \int_{B_R(0) \setminus B_{R/2}(0)} |v_t| \diff |\sigma_t| \diff t \fstop
\end{align*}
Passing to the limit as~$n \to \infty$ using~\eqref{eq:uconv}, we get
\begin{equation}\label{eq:laststep}
\int_0^T \int_{\R^d} \chi_R \psi_t \diff \sigma_t \diff t  \le \frac{4C}{R} \int_0^T \int_{B_R(0) \setminus B_{R/2}(0)} |v_t| \diff |\sigma_t| \diff t \fstop
\end{equation}
By monotone convergence theorem and by~\eqref{eq:pbd}, letting~$R \to \infty$ we obtain~$\int_0^T \int_{\R^d} \psi_t \diff \sigma_t \diff t \le 0$. By the arbitrariness of $\psi \geq 0$ the proof is complete.
\end{proof}

We conclude this section with a consequence of Proposition~\ref{prop:uniqu}.

\begin{corollary}
\label{cor:comp} Let~$v \colon [0,T] \times \R^d \to \R^d$,~$w \colon  [0,T] \times \R^d \to \R$ be Borel maps, and let~$\mu^1, \mu^2 \in \mcC([0,T]; \msM_+(\R^d))$ be satisfying the following assumptions:
\begin{align} \label{eq:w1loc12}
&\int_0^T \braket{ \sup_{B} |v_t| + \Lip(v_t, B) + \sup_B w_t^- + \sup_{\R^d} w_t^+ +\Lip(w_t, B) } \diff t < \infty \comma B \Subset \R^d \comma
\\ \label{eq:cer1}
&\vphantom{\int_{0}^{T}}  \partial_t \mu^1 + \div ( v \mu^1) \le w\mu^1 \text{ in } \mcD'((0,T) \times \R^d) \comma
\\ \label{eq:cer2}
&\vphantom{\int_{0}^{T}} \partial_t \mu^2 + \div ( v \mu^2) \ge w\mu^2 \text{ in } \mcD'((0,T) \times \R^d) \comma 
\\ \label{eq:pbd12}
& \int_0^T \int_{\R^d} \braket{|v_t| + |w_t|} \diff \mu^i_t \diff t < \infty\comma \, i=1,2 \fstop
\end{align}
Suppose moreover that~$\mu^1_0 \le \mu^2_0$. Then $\mu^1_t \le \mu^2_t$ for every $t \in [0,T]$.
\end{corollary}

\begin{proof} 
Let~$\sigma\eqdef \mu^1-\mu^2$ and observe that~$\sigma \in  \mcC([0,T]; \msM(\R^d))$. Obviously~\eqref{eq:w1loc} holds; by linearity,~\eqref{eq:cer1}, and~\eqref{eq:cer2}, also~\eqref{eq:cer} holds. Since~$|\sigma_t| \le \mu^1_t + \mu^2_t$ for every~$t \in [0,T]$, by~\eqref{eq:pbd12}, also~\eqref{eq:pbd} holds. If~$B \Subset \R^d$, then
\[
 |\sigma_t|(B)  \le \mu_t^1(B) + \mu_t^2(B) \le \mu_t^1(\R^d) + \mu_t^2(\R^d) \le \sup_{t \in [0,T]} \braket{\mu_t^1(\R^d) + \mu_t^2(\R^d)} < \infty\comma
\] 
where the last inequality comes from the continuity of~$[0,T] \ni t \mapsto \mu_t^i(\R^d) = \int 1 \diff \mu_t^i$, for~$i=1,2$. By Proposition~\ref{prop:uniqu}, we get that~$\sigma_t \le 0$ for every~$t \in [0,T]$, that is,~$\mu^1_t \le \mu^2_t$ for every $t \in [0,T]$.
\end{proof}

\begin{remark}[On the global bound from above on~$w$]\label{re:globalbound}
We remark that the global control $\int_0^T \sup_{\R^d} w_t^+ \diff t < \infty$ provided in assumption~\eqref{eq:w1loc} is only used in the last step of the proof of Proposition~\ref{prop:uniqu}. Indeed, setting
\[
A(R) \eqdef \int_0^T \sup_{\overline{B_{2R}(0)}} w_t^+ \diff t \comma \qquad \Gamma_R \eqdef T e^{A(R)} \comma \qquad R>0 \comma
\]
we notice that $A(R)$ and $\Gamma_{R}$ are finite for every~$R>0$ under the purely local version of~\eqref{eq:w1loc}, i.e., $\int_{0}^{T} \sup_{\overline{B _{2R} (0)} } |w_{t}| \, \diff t <\infty$, and are non-decreasing in~$R$. The chain of estimates~\eqref{e:estimate-varphinR} yields~$\|\phi^{n,R}\|_\infty \le \Gamma_R$ already after invoking~\eqref{eq:globalupper}. Every step of the proof up to~\eqref{eq:laststep} therefore goes through verbatim with~$C$ replaced by~$\Gamma_R$, as~$R$ is kept fixed and only the limit~$n \to \infty$ is taken. When trying to pass to the limit as $R \to \infty$ in the inequality
\[
\int_0^T \int_{\R^d} \chi_R \psi_t \diff \sigma_t \diff t  \le \frac{4\, \Gamma_R}{R} \int_0^T \int_{B_R(0) \setminus B_{R/2}(0)} |v_t| \diff |\sigma_t| \diff t \comma
\]
the conclusion of Proposition~\ref{prop:uniqu} holds as soon as
\begin{equation}\label{eq:sharpgrowth}
\liminf_{R \to \infty} \ \frac{\Gamma_R}{R} \int_0^T \int_{B_R(0) \setminus B_{R/2}(0)} |v_t| \diff |\sigma_t| \diff t = 0 \fstop
\end{equation}
By~\eqref{eq:pbd},~\eqref{eq:sharpgrowth} holds whenever~$\liminf_{R \to \infty} \Gamma_R/R < \infty$, that is, if there exist~$K>0$ and a sequence~$R_k \uparrow \infty$ such that
\begin{equation}\label{eq:loggrowth}
\int_0^T \sup_{\overline{B_{2R_k}(0)}} w_t^+ \diff t \le \log( K R_k ) \comma \qquad k \in \N \fstop
\end{equation}
Loosely speaking,~\eqref{eq:loggrowth} means that mass may be created at an unbounded rate at infinity, provided the rate grows slowly enough with respect to the~$1/R$ decay produced by the cutoff functions~$\chi_R$. One could then obtain the conclusion of~Proposition~\ref{prop:uniqu} replacing the condition $\int_0^T \sup_{\R^d} w_t^+ \diff t < \infty$ with the weaker~\eqref{eq:sharpgrowth}. Note that the same remark applies to Corollary~\ref{cor:comp}, whose proof only consists in applying Proposition~\ref{prop:uniqu} to~$\sigma \eqdef \mu^1 - \mu^2$.
\end{remark}

\section{The superposition principle under local assumptions}\label{s:repr}

This section is devoted to the proof of the following superposition principle featuring local boundedness and Lipschitz continuity of the fields $v$ and $w$. 

\begin{theorem}[Superposition principle under local assumptions]
\label{thm:supercone}
Let~$v\colon [0,T] \times \R^d \to \R^d$,~$w\colon [0,T] \times \R^d \to \R$ be Borel maps, and let~$\mu \in \mcC([0,T]; \msM_+(\R^d))$ be satisfying the following assumptions:
\begin{align} 
\label{eq:w1loccone}
&\int_0^T \braket{ \sup_{B} |v_t| + \Lip(v_t, B) + \sup_{B}|w_t| +\Lip(w_t, B) } \diff t < \infty \comma B \Subset \R^d \comma
\\ \label{eq:cercone}
&\vphantom{\int_{0}^{T}} \partial_t \mu + \div ( v \mu) = w\mu \text{ in } \mcD'((0,T) \times \R^d) \comma
\\ \label{eq:pbdcone}
& \int_0^T \int_{\R^d} \braket{|v_t|^2 + |w_t|^2} \diff \mu_t \diff t < \infty \fstop
\end{align}
Then there exists a probability measure~$\eta \in \mcP(\cic)$ such that:
\begin{enumerate}
\item $\eta$ is concentrated on curves~$\gamma \in\acic$ such that
\begin{equation}\label{eq:ode}
 x_\gamma'(t) = v_t(x_\gamma(t)) \comma \qquad k_\gamma'(t) = w_t(x_\gamma(t))\, k_\gamma(t) \qquad \text{for a.e.~} t \in O_\gamma \semicolon
\end{equation}
\item $h_t^2(\eta) = \mu_t$ for every~$t \in [0,T]$.
\end{enumerate}
\end{theorem}

The proof of Theorem~\ref{thm:supercone} builds upon the technical two-time representation of Lemma~\ref{le:twotime}, reported in Section~\ref{sub:representation-regular}. Section~\ref{sub:lemmata} collects some technical measurability results.

We start by introducing some auxiliary notation that will be used throughout this section.  If~$v \colon [0,T]\times\R^d\to\R^d$ is a Borel map satisfying the local assumption
\begin{equation}\label{eq:vloc}
\int_0^T \braket{ \sup_{B} |v_t| + \Lip(v_t, B) } \diff t < \infty \comma B  \Subset \R^d \comma
\end{equation}
for every~$(t,x) \in [0,T] \times \R^d$, there exists a unique maximal solution to the characteristic system
\[
X'(r) = v_r(X(r)) \comma \quad X(t)=x \comma
\]
defined on a relatively open subinterval~$I(t,x)$ of~$[0,T]$ containing~$t$ as relatively internal point, see e.g.~\cite[Lemma 8.1.4]{AmbGigSav08}; we denote by~$X_{t \to r}(x)$ its value at time~$r \in I(t,x)$.

 If~$w \colon [0,T]\times\R^d\to\R$ is a scalar Borel map satisfying the local assumption
\begin{equation}\label{eq:wloc}
\int_0^T\sup_{B} |w_t|  \diff t < \infty \comma B  \Subset \R^d \comma
\end{equation}
for every~$(s,x) \in [0,T] \times \R^d$ and~$\tau \in I(s,x)$, we define 
\[
W_{s \to \tau}(x) \eqdef e^{\int_s^\tau w_r(X_{s \to r}(x)) \diff r} \in (0, \infty) \fstop
\]
Note that, since~$r \mapsto X_{s \to r}(x)$ is continuous in~$[s \wedge \tau, s \vee \tau]$, it is bounded, so that~$[s \wedge \tau, s \vee \tau] \ni r \mapsto w_r(X_{s \to r}(x))$ is integrable by~\eqref{eq:wloc}; therefore~$W_{s \to \tau}(x)$ is well defined. We record the following cocycle rule: if~$b, c \in I(a,x)$, then~$I(b,X_{a \to b}(x))=I(a,x)$,~$X_{b \to c}(X_{a \to b}(x)) = X_{a \to c}(x)$ and
\begin{equation}\label{eq:cocycleW}
W_{a \to c}(x) = W_{a \to b}(x)\, W_{b \to c}(X_{a \to b}(x)) \,.
\end{equation}
In particular~$W_{a \to b}(x)\, W_{b \to a}(X_{a \to b}(x)) = W_{a\to a}(x) = 1$.

\subsection{Measurability results}
\label{sub:lemmata}

We collect in the next three lemmas some measurability results that are used repeatedly in the sequel and are independent of the rest of the section.
\begin{lemma}[Regularity of the flow]\label{le:flowreg}
Let~$v \colon [0,T] \times \R^d \to \R^d$ be a Borel map satisfying~\eqref{eq:vloc} and set
\begin{equation}\label{eq:defD}
\mcD \eqdef \set{ (s,x,\tau) \in [0,T] \times \R^d \times [0,T] : \tau \in I(s,x) } \fstop
\end{equation}
Then~$\mcD$ is relatively open in~$[0,T] \times \R^d \times [0,T]$ and the map~$(s,x,\tau) \mapsto X_{s \to \tau}(x)$ is continuous on~$\mcD$. In particular:
\begin{enumerate}
\item for every~$a,b \in [0,T]$ the set~$E_{a \to b} \eqdef \set{ x \in \R^d : b \in I(a,x) }$ is open in~$\R^d$;
\item $[0,T] \times \R^d \ni (s,x) \mapsto \inf I(s,x)$ is upper semicontinuous and~$[0,T] \times \R^d \ni (s,x) \mapsto \sup I(s,x)$ is lower semicontinuous.
\end{enumerate}
\end{lemma}

\begin{proof}
Let~$(s_0,x_0,\tau_0) \in \mcD$ and set~$Z_0(\tau) \eqdef X_{s_0 \to \tau}(x_0)$,~$\tau \in I(s_0, x_0)$. Since~$s_0,\tau_0 \in I(s_0,x_0)$ and $I (s_{0}, x_{0})$ is relatively open in~$[0,T]$, both~$s_0$ and~$\tau_0$ are relatively internal to~$I(s_0,x_0)$, so that we may fix a compact interval~$J_0 \subseteq I(s_0,x_0)$ having~$s_0$ and~$\tau_0$ as relatively internal points. Let
\[
K \eqdef \set{ Z_0(\tau) : \tau \in J_0 } \comma \qquad K_{\rho} \eqdef \set{ y \in \R^d : \dist(y,K) \le \rho } \text{ for } \rho>0 \comma
\]
which are compact, and set~$L \eqdef \int_0^T \Lip(v_r,K_{1}) \diff r < \infty$ by~\eqref{eq:vloc}. All the objects below depend on a parameter~$\eps \in (0, \tfrac12 e^{-L})$, which for the moment we fix arbitrarily. Since~$s_0$ is relatively internal to~$J_0$ and~$Z_0$ is continuous, there is~$\delta_\eps>0$ such that
\begin{align}
    \label{e:continuity}
    s \in [0,T] \text{ and } |s-s_0| + |x - x_0| < \delta_\eps \quad \Longrightarrow \quad s \in J_0 \text{ and } \big| x - Z_0(s) \big| < \eps \fstop
\end{align}
Fix a pair~$(s,x)$ satisfying the left hand side of~\eqref{e:continuity} and set~$Z(\tau) \eqdef X_{s \to \tau}(x)$,~$\tau \in I(s,x)$. We claim that
\begin{equation}\label{eq:claimflow}
J_0 \subseteq I(s,x) \qquad \text{and} \qquad \sup_{\tau \in J_0} |Z(\tau)-Z_0(\tau)| \le \eps\, e^L \fstop
\end{equation}
To prove~\eqref{eq:claimflow}, consider the set
\[
G \eqdef \set{ \tau \in J_0 \cap I(s,x) : Z\big([\tau \wedge s, \tau \vee s]\big) \subseteq K_1 } \comma
\]
which is nonempty, since~$s \in G$ because~$Z(s)=x$ and~$|x-Z_0(s)|<\eps<1$, and which is an interval, since~$\tau \in G$ implies that every point between~$s$ and~$\tau$ belongs to~$G$. We show that~$G$ is relatively open and closed in~$J_0$; being nonempty, it then coincides with the interval~$J_0$.

Fix~$\tau \in G$. Both~$Z$ and~$Z_0$ solve the characteristic system~$Y' = v_r(Y)$ on~$[\tau \wedge s, \tau \vee s] \subseteq I(s,x) \cap I(s_0, x_0)$, where they take values in~$K_1$; hence
\[
|Z(\rho)-Z_0(\rho)| \le |Z(s)-Z_0(s)| + \int_{\rho \wedge s}^{\rho \vee s} \Lip(v_r,K_1)\, |Z(r)-Z_0(r)| \diff r \fstop
\]
Gr\"onwall's inequality then gives
\begin{equation}\label{eq:gronwallflow}
|Z(\rho) - Z_0(\rho)| \le \eps\, e^{L} < \tfrac12 \comma \qquad \rho \in [\tau \wedge s, \tau \vee s] \comma
\end{equation}
so that~$Z(\rho) \in K_{1/2}$ there; in particular~$Z(\tau) \in K_{1/2}$, which lies in the interior of~$K_1$. Since~$\tau \in I(s,x)$,~$I(s,x)$ is relatively open in~$[0,T]$, and~$Z$ is continuous, there is a relatively open (in~$J_0$) subinterval~$N \ni \tau$ with~$N \subseteq I(s,x)$ and~$Z(N) \subseteq K_1$. For every~$\tau' \in N$ the segment~$[\tau' \wedge s, \tau' \vee s]$ is contained in~$[\tau \wedge s, \tau \vee s] \cup N$, where~$Z$ takes values in~$K_1$; hence~$\tau' \in G$, so that~$G$ is relatively open. If instead~$\tau_n \in G$ and~$\tau_n \to \tau \in J_0$, then $ (\tau \wedge s , \tau \vee s) \subseteq G$. Thus,~$Z$ takes values in~$K_{1/2}$ on~$[\tau \wedge s, \tau \vee s] \setminus \set{\tau}$ by~\eqref{eq:gronwallflow}, with~$|Z'(r)| \le \sup_{K_1}|v_r| \in L^1([0,T])$ there; therefore~$Z$ is absolutely continuous up to~$\tau$, takes values in the compact set~$K_{1/2}$, and, by maximality of~$I(s,x)$, is defined at~$\tau$, so that~$\tau \in G$: hence~$G$ is relatively closed. This proves~\eqref{eq:claimflow}.

Since~$\tau_0$ is relatively internal to~$J_0$, we may fix any relatively open (in~$[0,T]$) neighbourhood~$U$ of~$\tau_0$ with~$U \subseteq J_0$. If~$(s,x,\tau) \in [0,T]\times \R^d \times [0,T]$ satisfies~$|s-s_0|< \delta_\eps/2$,~$|x-x_0|< \delta_\eps/2$ and~$\tau \in U$, then~$(s,x)$ fulfills the left hand side of~\eqref{e:continuity}, and~\eqref{eq:claimflow} gives~$\tau \in J_0 \subseteq I(s,x)$. Thus~$\mcD$ contains the relatively open neighbourhood~$\set{ (s,x,\tau) : |s-s_0|<\delta_\eps/2,\ |x-x_0|<\delta_\eps/2,\ \tau \in U }$ of~$(s_0,x_0,\tau_0)$, i.e.~it is relatively open. 

For the continuity of the flow, fix~$\vartheta \in (0,1)$; we now specialize~$\eps$ to a value~$0<\eps<\tfrac{\vartheta}{2} e^{-L} <\tfrac{1}{2} e^{-L}$ and, by the continuity of~$Z_0$ at~$\tau_0$, we may replace~$U$ with~$U_\vartheta$ so that~$\big| X_{s_0 \to \tau}(x_0) - X_{s_0 \to \tau_0}(x_0) \big| < \tfrac\vartheta2$ for every~$\tau \in U_\vartheta$. Then, for every~$(s,x,\tau) \in [0,T] \times \R^d \times [0,T]$ with~$|s-s_0|<\delta_\eps/2$,~$|x-x_0|<\delta_\eps/2$ and~$\tau \in U_\vartheta$, the second part of~\eqref{eq:claimflow} gives
\[
\big| X_{s \to \tau}(x) - X_{s_0 \to \tau_0}(x_0) \big| \le \eps\, e^L + \big| X_{s_0 \to \tau}(x_0) - X_{s_0 \to \tau_0}(x_0) \big| < \vartheta \fstop
\]
As~$\vartheta \in (0,1)$ and~$(s_0,x_0,\tau_0) \in \mcD$ are arbitrary, this proves the continuity of the flow on~$\mcD$.

Item~$(1)$ follows because~$E_{a \to b}$ is the section of the open set~$\mcD$ at~$(a,\cdot,b)$. As for item~$(2)$, if~$\sup I(s_0,x_0) > c$ then~$(s_0,x_0,\tau) \in \mcD$ for some~$\tau>c$; since~$\mcD$ is relatively open in~$[0,T] \times \R^d \times [0,T]$, we can find some~$\delta>0$ such that, if~$(s,x) \in [0,T] \times \R^d$ are such that~$|s_0-s|+|x-x_0|< \delta$, then~$(s,x,\tau) \in \mcD$, giving in particular that~$\sup I(s,x)>c$ for every such~$(s,x)$. The argument for the infimum is symmetric.
\end{proof}

\begin{lemma}[Measurability of the derived objects]
\label{le:borelmaps}
 Let~$v \colon [0,T] \times \R^d \to \R^d$,~$w \colon [0,T] \times \R^d \to \R$ be Borel maps satisfying the following assumption
\[
\int_0^T \braket{ \sup_{B} |v_t| + \Lip(v_t, B) + \sup_{B} |w_t|} \diff t < \infty \comma B \Subset \R^d \comma
\]
and let~$\mathcal{D}$ be as in~\eqref{eq:defD}. Then, the following holds:
\begin{enumerate}
\item the map $(s,x,\tau) \mapsto W_{s \to \tau}(x)$ is Borel on~$\mcD$;
\item for every Borel map~$f \colon [0,T] \times \R^d \to [0,\infty]$ the function
\[
[0,T] \times \R^d \ni (s,x) \longmapsto \int_{I(s,x)} f_u\big( X_{s \to u}(x) \big)\, W_{s \to u}(x) \diff u \in [0,\infty]
\]
is Borel.
\end{enumerate}
\end{lemma}

\begin{proof}
We prove item $(1)$. The map~$\Xi(s,x,r) \eqdef \big( r, X_{s \to r}(x) \big)$ is continuous on~$\mcD$ by Lemma~\ref{le:flowreg}, so that~$g \eqdef w \circ \Xi$ is Borel on~$\mcD$. We extend~$g$ by~$0$ outside~$\mcD$, which is a Borel set. The function
\[
(s,x,\tau,r) \longmapsto g^{\pm}(s,x,r)\, \car_{\set{ s \wedge \tau \le r \le s \vee \tau }}(r)
\]
is Borel and non-negative, so by Tonelli's theorem~$G^\pm(s,x,\tau) \eqdef \int_{s \wedge \tau}^{s \vee \tau} g^\pm(s,x,r) \diff r$ is a Borel function defined on~$[0,T] \times \R^d \times [0,T]$ with values in~$[0,\infty)$, as we have already observed that~$[s \wedge \tau, s \vee \tau] \ni r \mapsto w_r(X_{s \to r}(x))$ is integrable by~\eqref{eq:wloc}.  Since
\[
\int_s^\tau w_r(X_{s \to r}(x)) \diff r = \sgn(\tau-s) \braket{G^+ (s,x,\tau)- G^-(s,x,\tau)}  \quad \text{on } \mcD  \comma
\]
we deduce that~$W$ is Borel.

As for item $(2)$, the function~$(s,x,u) \mapsto f_u(X_{s \to u}(x))\, W_{s \to u}(x)\, \car_{\mcD}(s,x,u)$ is Borel by Lemma~\ref{le:flowreg} and by item~$(1)$, and it is non-negative. Thus, item $(2)$ follows from Tonelli's theorem.
\end{proof}

The last lemma is a general measure-theoretic fact, which we shall use to lift measurability from the values of a curve to the curve itself.

\begin{lemma}[Measurability of curve-valued maps]\label{le:curvemeas}
Let~$(\mcY, \mssd_{\mcY})$ be a separable metric space, let~$G$ be a Borel subset of a metric space and let~$F \colon G \to \mcC([0,T];\mcY)$ be such that~$g \mapsto F(g)_\tau$ is Borel for every~$\tau \in [0,T]$. Then~$F$ is Borel.
\end{lemma}

\begin{proof}
Since~$(\mcY, \mssd_{\mcY})$ is separable, so is~$\mcC([0,T];\mcY)$. Let~$D \eqdef \Q \cap [0,T]$. For every~$\gamma \in \mcC([0,T];\mcY)$ the function
\[
g \longmapsto \sup_{\tau \in [0,T]} \mssd_{\mcY}\big( F(g)_\tau, \gamma_\tau \big) = \sup_{\tau \in D} \mssd_{\mcY}\big( F(g)_\tau, \gamma_\tau \big)
\]
is Borel, being a countable supremum of Borel functions. Hence, the preimage under~$F$ of every open ball is Borel, and so is the preimage of every open set, the latter being a countable union of open balls by separability.
\end{proof}

\subsection{Key tool: Two-time representation}
\label{sub:representation-regular}

The following two-time representation is the key step. Recall that~$X$ and~$W$ have been introduced at the beginning of this section and that the sets~$E_{a \to b}$ are defined in Lemma~\ref{le:flowreg}.

\begin{lemma}[Two-time representation]\label{le:twotime}
Let~$v \colon [0,T] \times \R^d \to \R^d$,~$w \colon [0,T] \times \R^d \to \R$ be Borel maps, and let~$\mu \in \mcC([0,T]; \msM_+(\R^d))$ be satisfying~\eqref{eq:w1loccone},~\eqref{eq:cercone}, and
\begin{align} 
\label{eq:pbdlocle}
& \int_0^T \int_{\R^d} \braket{|v_t| + |w_t|} \diff \mu_t \diff t < \infty \fstop
\end{align}
Then, for every~$s,t \in [0,T]$ it holds
\begin{equation}\label{eq:twotime}
(X_{s \to t})_\sharp \paren{ W_{s\to t} \, \mu_s \restr{E_{s \to t}} } = \mu_t \restr{ E_{t \to s}}\fstop
\end{equation}
\end{lemma}

\begin{proof}
We first collect a few facts, valid for arbitrary times. For~$a, b \in [0,T]$ the set $E_{a \to b}$ is open, hence Borel, by Lemma~\ref{le:flowreg}(1). Since the maximal interval is constant along each trajectory, we have that
\begin{equation}\label{eq:bijEab}
X_{a \to b} \colon E_{a \to b} \to E_{b \to a} \quad \text{is a bijection with inverse } X_{b \to a} \fstop
\end{equation}
Indeed, if~$x \in E_{a \to b}$ then~$y \eqdef X_{a \to b}(x)$ satisfies~$I(b,y)=I(a,x) \ni a$, i.e.~$y \in E_{b \to a}$; and symmetrically, if~$y \in E_{b \to a}$ then~$X_{b \to a}(y) \in E_{a \to b}$.

We recall that for every $(a, x, b) \in \mathcal{D}$ (the latter is defined in Lemma~\ref{le:flowreg}), the map
\[
W_{a \to b}(x) \eqdef e^{\int_a^b w_r(X_{a \to r}(x)) \diff r} \in (0,\infty)
\]
is well-defined thanks to~\eqref{eq:w1loccone}.

We also observe that, for every~$\phi \in \mcC_b^\infty([0,T] \times \R^d)$ and every~$(a,x, b) \in \mathcal{D}$ the map
\[
\tau \mapsto \phi_\tau(X_{a \to \tau}(x))\, W_{a \to \tau}(x)
\]
is absolutely continuous on the closed interval with endpoints~$a$ and~$b$, with a.e.~derivative in time given by
\begin{align*}
g^a_\phi(\tau,x) \eqdef  W_{a \to \tau}(x) \biggl [ \partial_\tau \phi_\tau(X_{a \to \tau}(x)) + \langle \nabla \phi_\tau(X_{a \to \tau}(x)), v_\tau(X_{a \to \tau}(x)) \rangle
+ \phi_\tau( X_{a \to \tau}(x)) w_\tau(X_{a \to \tau}(x)) \biggr ]\fstop
\end{align*}

We divide the proof of~\eqref{eq:twotime} in two claims according to the mutual position of $s, t \in [0, T]$.

\medskip

\noindent\paragraph{Claim~1} For every~$0 \le s \le t \le T$ it holds
\begin{equation}\label{eq:halfonefull}
(X_{s \to t})_\sharp \left ( W_{s \to t} \mu_s \restr{E_{s \to t}} \right ) \le \mu_t \restr{E_{t \to s}}.
\end{equation}

\smallskip

\noindent\paragraph{Proof of Claim~1} We can assume~$s \ne t$; to prove~\eqref{eq:halfonefull}, we fix~$s,t \in [0,T]$ with~$s<t$ and abbreviate
\[
E \eqdef E_{s \to t} = \set{ z \in \R^d : t \in I(s,z) } \comma \qquad F \eqdef E_{t \to s} = \set{ y \in \R^d : s \in I(t,y) } \comma
\]
so that, by~\eqref{eq:bijEab},~$X_{s \to t} \colon E \to F$ is a bijection with inverse~$X_{t \to s}$. For~$\tau \in [s,t]$ set
\[
\lambda_\tau \eqdef (X_{s \to \tau})_\sharp \paren{ W_{s \to \tau}\, \mu_s \restr{E} } \fstop
\]
Proving~\eqref{eq:halfonefull} is then equivalent to showing
\begin{equation}\label{eq:halfone}
\lambda_t \le \mu_t \restr{F} \fstop
\end{equation}
To prove~\eqref{eq:halfone}, for~$n \in \N$, set
\begin{align*}
E_n &\eqdef \set{ z \in E : |X_{s \to r}(z)| \le n \text{ for every } r \in [s,t] } \comma
\\
\lambda^n_\tau &\eqdef (X_{s \to \tau})_\sharp \paren{ W_{s \to \tau}\, \mu_s \restr{E_n} } \comma \qquad &&\tau \in [s,t] \comma
\\
M_\tau^n &\eqdef \sup_{\overline{B_n(0)}} |w_\tau| \comma \qquad &&\tau \in [0,T] \comma
\\
w_\tau^n &\eqdef w_\tau \wedge M_\tau^n \comma \qquad &&\tau \in [0,T] \fstop
\end{align*}
We claim that~$v, w^n, \lambda^n$, and~$\mu$ satisfy the hypotheses of Corollary~\ref{cor:comp} on~$[s,t]$. Clearly~$v$ and~$w^n$ are Borel and satisfy~\eqref{eq:w1loc12}. Since~$w^n \le w$ and~$|w^n| \le |w|$, we have that~\eqref{eq:cer2} and~\eqref{eq:pbd12} are satisfied for~$\mu$. By assumption, we also have that~$\mu \in \mcC([s,t]; \msM_+(\R^d))$. Clearly~\eqref{eq:pbd12} for~$\lambda^n$ is satisfied because of the bounds on~$v$ and~$w$ coming from~\eqref{eq:w1loccone} and the fact that~$|X_{s \to \tau}(x)| \le n$ for every~$\tau \in [s,t]$ and every~$x \in E_n$.

Let us prove that $\lambda^{n} \in \mcC([s,t]; \msM_+(\R^d))$. Let~$\phi \in \Cb(\R^d)$. For every~$\tau_0 \in [s,t]$, we have
\begin{align*}
 \lim_{ \substack{\tau \to \tau_0 \\ \tau \in [s,t] }} \int_{\R^d} \phi \diff \lambda_\tau^n&= \lim_{\substack{\tau \to \tau_0 \\ \tau \in [s,t] }} \int_{\R^d} \phi(X_{s \to \tau}(x))  W_{s \to \tau}(x) \diff \mu_s \restr{E_n} (x) 
\\
&= \int_{\R^d} \phi(X_{s \to \tau_0}(x))  W_{s \to \tau_0}(x) \diff \mu_s \restr{E_n}(x) = \int_{\R^d} \phi \diff \lambda^n_{\tau_0} \comma
\end{align*}
where we have used the dominated convergence theorem with constant dominant function
\[
\|\phi\|_\infty e^{\int_0^T \sup_{\overline{B_n(0)}}|w_r| \diff r}  \in L^1(\R^d, \mu_s \restr{E_n})
\]
together with the continuity of~$\phi$,~$[s,t] \ni \tau \mapsto X_{s \to \tau}(x)$, and~$[s,t] \ni \tau \mapsto  W_{s \to \tau}(x)$. Hence,~$\lambda^n \in \mcC([s,t]; \msM_+(\R^d))$. 

To prove that~$\lambda^n$ satisfies~\eqref{eq:cer1}, it is enough to notice that for every~$\phi \in \mcC_c^\infty((s,t) \times \R^d)$ we have
\[
\int_s^t\int_{\R^d} |g^s_\phi(\tau,x)| \diff \mu_s \restr{E_n} (x) \diff \tau < \infty
\]
because of the bounds on~$v$ and~$w$ coming from~\eqref{eq:w1loccone} and the fact that~$|X_{s \to \tau}(x)| \le n$ for every~$\tau \in [s,t]$ and every~$x \in E_n$, and
\begin{align*}
 \int_{\R^{d}} \phi_{t} (x) \, \diff \lambda^{n}_{t} (x) - \int_{\R^{d}} \phi_{s} (x) \, \diff \lambda^{n}_{s} (x)  &= \int_{\R^d} \phi_t(X_{s \to t}(x)) W_{s \to t}(x)\diff \mu_s \restr{E_n}(x) - \int_{\R^d} \phi_s(x) \diff \mu_s \restr{E_n} (x)
\\
&= \int_{\R^d} \paren{ \phi_t(X_{s \to t}(x)) W_{s \to t}(x) - \phi_s(X_{s \to s}(x)) W_{s \to s}(x)  } \diff \mu_s \restr{E_n} (x)
\\
& = \int_{\R^d} \int_s^t g^s_\phi(\tau,x) \diff \tau  \diff \mu_s \restr{E_n} (x)
\\
& =\int_s^t \int_{\R^d}  g^s_\phi(\tau,x)  \diff \mu_s \restr{E_n} (x) \diff \tau
\\
& =\int_s^t \int_{\R^d}  \paren{ \partial_\tau \phi + \langle \nabla \phi, v_\tau \rangle + w_\tau \phi} \diff \lambda_\tau^n \diff \tau  
\\
& =\int_s^t \int_{\R^d}  \paren{ \partial_\tau \phi + \langle \nabla \phi, v_\tau \rangle + w^n_\tau \phi} \diff \lambda_\tau^n \diff \tau  \comma
\end{align*}
where we have used Fubini's theorem and the fact that~$w_\tau = w_\tau^n$ in~$\supp(\lambda_\tau^n)$ for every~$\tau \in [s,t]$. Since~$\lambda_s^n = \mu_s \restr{E_n} \le \mu_s$, Corollary~\ref{cor:comp} gives that
\[
\lambda^n_\tau \le \mu_\tau \comma \qquad \tau \in [s,t] \comma\ n \in \N \fstop
\]
Since~$E_n \uparrow E$ as~$n \to \infty$, monotone convergence gives~$\lambda^n_\tau \uparrow \lambda_\tau$, whence~$\lambda_\tau \le \mu_\tau$ for every~$\tau \in [s,t]$. In particular~$\lambda_t \le \mu_t$. Since~$\lambda_t$ is concentrated on~$X_{s \to t}(E)=F$, this proves~\eqref{eq:halfone}.

\medskip
\noindent\paragraph{Claim~2} For every~$0 \le t \le s \le T$ it holds
\begin{equation}\label{eq:halftwofull}
(X_{s \to t})_\sharp \left ( W_{s \to t} \,  \mu_s \restr{E_{s \to t}} \right ) \le \mu_t \restr{E_{t \to s}}.
\end{equation}

\smallskip

\noindent\paragraph{Proof of Claim~2} We reduce to Claim~1 by reversing time. For~$\theta \in [0,T]$ set
\[
\hat v_\theta \eqdef -v_{T-\theta} \comma \qquad \hat w_\theta \eqdef -w_{T-\theta} \comma \qquad \hat\mu_\theta \eqdef \mu_{T-\theta} \fstop
\]
It is clear that~$(\hat v, \hat w, \hat \mu)$ satisfies the assumptions of the present lemma. Now let~$\hat X$,~$\hat I$,~$\hat W$ and~$\hat E$ be associated with~$(\hat v, \hat w, \hat\mu)$ as~$X$,~$I$,~$W$ and~$E$ are with~$(v,w,\mu)$. For every~$x \in \R^d$ and~$a,b \in [0,T]$,
\begin{equation}\label{eq:dictionary}
\hat X_{a \to b}(x) = X_{T-a \to T-b}(x) \comma \qquad \hat I(a,x) = \set{ T-r : r \in I(T-a,x) } \comma
\end{equation}
\begin{equation}\label{eq:dictionaryE}
\hat E_{a \to b} = E_{T-a \to T-b} \fstop
\end{equation}
Indeed~$X_{T-a \to T-a}(x)=x$ and
\[
\tfrac{\diff}{\diff b} X_{T-a \to T-b}(x) = -v_{T-b}\big(X_{T-a \to T-b}(x)\big) = \hat v_b\big(X_{T-a \to T-b}(x)\big) \comma
\]
which gives the first identity; the second follows, and \eqref{eq:dictionaryE} is a restatement of it, because~$b \in \hat I(a,x)$ if and only if~$T-b \in I(T-a,x)$. Consequently
\begin{equation}\label{eq:dictionaryW}
\hat W_{a \to b}(x) = e^{\int_a^b \hat w_r(\hat X_{a \to r}(x)) \diff r} = W_{T-a \to T-b}(x) \comma
\end{equation}
as one sees substituting~$\rho = T-r$, which turns
\[
 \int_a^b \big(-w_{T-r}\big)\big(X_{T-a \to T-r}(x)\big) \diff r \quad \text{into} \quad \int_{T-a}^{T-b} w_\rho\big(X_{T-a \to \rho}(x)\big) \diff \rho \fstop 
\]

Let~$0 \le t \le s \le T$ be given and set~$s' \eqdef T-s$ and~$t' \eqdef T-t$, so that~$0 \le s' \le t' \le T$. Applying Claim~1 to~$(\hat v, \hat w, \hat \mu)$ and to the pair~$s' \le t'$ we obtain
\[
(\hat X_{s' \to t'})_\sharp \left ( \hat W_{s' \to t'}\, \hat\mu_{s'} \restr{\hat E_{s' \to t'}} \right ) \le \hat\mu_{t'} \restr{\hat E_{t' \to s'}} \fstop
\]
By~\eqref{eq:dictionary},~\eqref{eq:dictionaryE} and~\eqref{eq:dictionaryW},
\begin{align*}
\hat X_{s' \to t'} &= X_{s \to t} \comma &\hat W_{s' \to t'} &= W_{s \to t} \comma &\hat\mu_{s'} &= \mu_s \comma
\\
\hat\mu_{t'} &= \mu_t \comma &\hat E_{s' \to t'} &= E_{s \to t} \comma &\hat E_{t' \to s'} &= E_{t \to s} \comma
\end{align*}
so that the inequality above is precisely~\eqref{eq:halftwofull}.

\medskip

\noindent\paragraph{Conclusion} Claims~1 and~2 together state that
\begin{equation}\label{eq:halfany}
(X_{s \to t})_\sharp \left ( W_{s \to t} \, \mu_s \restr{E_{s \to t}} \right ) \le \mu_t \restr{E_{t \to s}} \qquad \text{for every } s,t \in [0,T] \fstop
\end{equation}
It remains to prove the converse; let~$s,t \in [0,T]$ be fixed. Applying~\eqref{eq:halfany} to the pair~$(t,s)$ we have
\begin{equation}\label{eq:halfonefullts}
(X_{t \to s})_\sharp \left ( W_{t \to s}\, \mu_t \restr{E_{t \to s}} \right ) \le \mu_s \restr{E_{s \to t}} \fstop
\end{equation}
Multiplying by~$W_{s \to t}$ both sides and applying the push forward by the map~$X_{s \to t}$ gives
\[
(X_{s \to t})_\sharp \braket{ W_{s \to t} \cdot (X_{t \to s})_\sharp \left ( W_{t \to s} \, \mu_t \restr{E_{t \to s}} \right )} \le (X_{s \to t})_\sharp \left ( W_{s \to t} \,  \mu_s \restr{E_{s \to t}} \right ) \fstop
\]
Thanks to~\eqref{eq:cocycleW} the left-hand side above rewrites as
\[
\braket{W_{s \to t } \circ X_{t \to s}} \cdot \braket{X_{s \to t} \circ X_{t \to s}}_\sharp \left ( W_{t \to s}\,  \mu_t \restr{E_{t \to s}} \right ) = W_{t \to t} \,  \mu_t \restr{E_{t \to s}} = \mu_t \restr{E_{t \to s}} \fstop
\]
Thus, we deduce
\[
\mu_t \restr{E_{t \to s}} \le (X_{s \to t})_\sharp \left ( W_{s \to t} \, \mu_s \restr{E_{s \to t}} \right )\comma
\]
which is the converse to~\eqref{eq:halfany}.
\end{proof}

\subsection{Proof of Theorem~\ref{thm:supercone}} We are now in a position to prove Theorem~\ref{thm:supercone}. In what follows, we are going to use the notation for curves in the cone introduced in Section~\ref{ssec:curvcone}, and in particular the objects defined in~\eqref{eq:notcurves}.

 We may assume that~$\mu$ does not vanish identically. Since~$\int_0^T \mu_t(\R^d) \diff t < \infty$ by the continuity of~$\mu$, the Cauchy--Schwarz inequality gives
\[
\int_0^T \int_{\R^d} \braket{|v_t| + |w_t|} \diff \mu_t \diff t \le \sqrt{2} \paren{ \int_0^T \int_{\R^d} \braket{|v_t|^2 + |w_t|^2} \diff \mu_t \diff t }^{1/2} \paren{ \int_0^T \mu_t(\R^d) \diff t }^{1/2} \comma
\]
so that~\eqref{eq:pbdcone} implies~\eqref{eq:pbdlocle}. Together with~\eqref{eq:w1loccone} and~\eqref{eq:cercone}, this shows that~$(v,w, \mu)$ satisfies the assumptions of Lemma~\ref{le:twotime}, whose notation we keep throughout. We define
\begin{align*}
\Theta &\eqdef \int_0^T (\delta_t \otimes \mu_t) \diff t \in \msM_+([0,T] \times \R^d) \comma
\\
\mfm(s,x) &\eqdef \tfrac{ \inf I(s,x) + \sup I(s,x) }{2} \in [0,T] \comma \qquad (s,x) \in [0,T] \times \R^d \comma
\\
S &\eqdef \set{ (s,x) \in [0,T] \times \R^d : \mfm(s,x) = s } \subset [0,T] \times \R^d \comma
\\
\Psi &\colon [0,T]\times\R^d \to S \comma \qquad \Psi(s,x) \eqdef \big( \mfm(s,x),\, X_{s \to \mfm(s,x)}(x) \big) \comma
\\
C(s,x) &\eqdef \int_{I(s,x)} W_{s \to u}(x) \diff u \in (0,\infty] \comma \qquad (s,x) \in [0,T] \times \R^d \fstop
\end{align*}
We observe that~$\Theta$ is a finite (non-zero, non-negative Borel) measure since~$\Theta([0,T] \times \R^d) = \int_0^T \mu_t(\R^d) \diff t \in (0,\infty)$, by the continuity of~$\mu$. The map~$\mfm$ is Borel by Lemma~\ref{le:flowreg}(2) and obviously satisfies~$\mfm(s,x) \in I(s,x)$ for every~$(s,x) \in [0,T] \times \R^d$; this remark and the first part of Lemma~\ref{le:flowreg} give the measurability and well definition of the set~$S$ and the map~$\Psi$. On the other hand, the map~$C$ is Borel by Lemma~\ref{le:borelmaps}(2) applied with~$f \equiv 1$. By~\eqref{eq:cocycleW}, for every~$(s,x) \in [0,T] \times \R^d$ and~$r \in I(s,x)$
\begin{equation}\label{eq:Ccocycle}
C(s,x) = \int_{I(s,x)} W_{s \to r}(x)\, W_{r \to u}(X_{s \to r}(x)) \diff u = W_{s \to r}(x)\, C(r,X_{s \to r}(x)) \fstop
\end{equation}
We further observe that
\begin{align*}
\int_{\R^d} C(t,x) \diff \mu_t(x) &= \int_0^T \int_{\R^d} W_{t \to u}(x)\, \car_{  E_{t\to u}  }(x) \diff \mu_t(x) \diff u
\\
&= \int_0^T \mu_u\big(  E_{u \to t}  \big) \diff u \le \int_0^T \mu_u(\R^d) \diff u < \infty \comma
\end{align*}
by the definition of~$C$, by Tonelli's theorem, and by~\eqref{eq:twotime} applied to the pair~$(t,u)$ and tested with the constant function~$1$. In particular, we deduce
\begin{equation}\label{eq:Cfinite}
C(t,x) \in (0,\infty) \quad \text{for~$\mu_t$-a.e.~} x \in \R^d \fstop
\end{equation}
 By~\eqref{eq:Ccocycle} applied to~$(s,x) \in [0,T] \times \R^d$ and~$r=\mfm(s,x) \in I(s,x)$, we get that
\[
C(s,x) = W_{s \to \mfm(s,x)}(x) (C \circ \Psi)(s,x)
\]
which implies that
\begin{equation}\label{eq:Cfiniterho}
C \in (0,\infty) \quad \Psi_\sharp \Theta\text{-a.e.}
\end{equation}
This justifies the definition of the following non-negative, non-zero (possibly infinite) Borel measure on~$[0,T] \times \R^d$
\[
\rho \eqdef C^{-1}\Psi_\sharp \Theta \fstop
\]
Since~$\Psi_\sharp \Theta$ is a finite measure and~$C \in (0,\infty)$~$\Psi_\sharp \Theta$-a.e., then~$\rho$ is~$\sigma$-finite. Hence, there exists a Borel function~$\theta \colon [0,T] \times \R^d \to (0,\infty)$ such that~$\int \theta^2 \diff \rho =1$. For every~$(s,x) \in [0,T] \times \R^d$, we define the curve $\mssT(s, x) \colon [0, T] \to \mfC[\R^{d}]$ as
\[
\mssT(s,x)_\tau \eqdef
\begin{cases}
\big[\, X_{s \to \tau}(x)\,,\ W_{s \to \tau}^{1/2}(x)/\theta(s,x) \,\big] & \text{if } \tau \in I(s,x)\comma
\\[3pt]
\mfo & \text{otherwise,}
\end{cases}
\qquad \tau \in [0,T]
\]
and we set
\begin{equation}\label{eq:defenergy}
\mcE(s,x) \eqdef \int_{I(s,x)} \braket{ \big|v_\tau(X_{s \to \tau}(x))\big|^2 + \tfrac14 \big|w_\tau(X_{s \to \tau}(x))\big|^2 } W_{s \to \tau}(x) \diff \tau \in [0,\infty] \comma
\end{equation}
which is a Borel function of~$(s,x)$ by Lemma~\ref{le:borelmaps}(2) applied with~$f \eqdef |v|^2 + \tfrac14 |w|^2$.

\medskip

\noindent\paragraph{Claim~1} We have
\begin{equation}\label{eq:energyae}
\mcE(s,x) < \infty \qquad \text{for~$\rho$-a.e.~} (s,x) \in [0,T] \times \R^d \fstop
\end{equation}

\smallskip

\noindent\paragraph{Proof of Claim~1} We test~\eqref{eq:twotime} for the pair~$s , \tau \in [0, T]$ with the non-negative Borel function~$|v_\tau|^2 + \tfrac14 |w_\tau|^2$. Using~$\mu_\tau \restr{ E_{\tau \to s} } \le \mu_\tau$ we get
\[
\int_{E_{s \to \tau}} \braket{ |v_\tau|^2 \big(X_{s \to \tau}(x)\big)+ \tfrac14 |w_\tau|^2\big(X_{s \to \tau}(x)\big) }\, W_{s \to \tau}(x) \diff \mu_s(x) \le \int_{\R^d} \braket{ |v_\tau|^2 + \tfrac14 |w_\tau|^2 } \diff \mu_\tau \fstop
\]
Integrating in~$\tau \in (0,T)$ and using Tonelli's theorem, we infer from~\eqref{eq:pbdcone} that
\begin{equation}\label{eq:energybound}
\int_{\R^d} \mcE(s,x) \diff \mu_s(x) \le \int_0^T \int_{\R^d} \braket{ |v_\tau|^2 + \tfrac14 |w_\tau|^2 } \diff \mu_\tau \diff \tau < \infty \comma \qquad s \in [0,T] \fstop
\end{equation}
 In particular~$\mcE < \infty$~$\Theta$-a.e.; moreover, exactly as in~\eqref{eq:Ccocycle}, the invariance of~$I$ and~$X$ along trajectories and~\eqref{eq:cocycleW} give
\begin{equation}\label{eq:energycocycle}
\mcE\big(r, X_{s \to r}(x)\big) = W^{-1}_{s \to r}(x)\, \mcE(s,x) \comma \qquad r \in I(s,x) \fstop
\end{equation}
 Applying this formula to~$(s,x) \in [0,T] \times \R^d$ and~$r=\mfm(s,x) \in I(s,x)$, we get that
\[
\mcE(s,x) = W_{s \to \mfm(s,x)}(x) (\mcE \circ \Psi)(s,x)
\]
which implies that
\begin{equation}\label{eq:Efiniterho}
\mcE < \infty \quad \Psi_\sharp \Theta\text{-a.e.}
\end{equation}
hence the claim. 

\medskip

\noindent\paragraph{Claim~2} For~$\rho$-a.e.~$(s,x) \in [0,T] \times \R^d$ the curve~$\mssT(s,x)$ is continuous and the map~$\mssT \colon [0,T] \times \R^d \to \cic$ is Borel. Consequently,~$\eta \eqdef \mssT_{\sharp} (\theta^{2} \rho)$ belongs to~$\mcP(\cic)$.

\smallskip

\noindent\paragraph{Proof of Claim~2} We fix~$(s,x)$ such that~$\mcE(s,x)<\infty$, which holds for~$\rho$-a.e.~$(s, x) \in [0, T] \times \R^{d}$ by Claim~1. We abbreviate
\[
J \eqdef I(s,x) \comma \qquad a \eqdef \inf J \comma \qquad b \eqdef \sup J \comma \qquad \mathfrak{w}(\tau) \eqdef W^{1/2}_{s \to \tau}(x) \comma \quad \tau \in J \fstop
\]
Since~$\tau \mapsto X_{s \to \tau}(x)$ and~$\tau \mapsto W_{s \to \tau}(x)$ are continuous in~$J$ and~$W_{s\to\tau}(x)>0$ in $J$, the curve~$\mssT(s,x)$ is continuous at every point of~$J$, and it is constantly equal to~$\mfo$ in the interior of~$[0,T] \setminus \overline{J}$. It therefore suffices to prove that
\begin{equation}\label{eq:vanishatends}
\lim_{\substack{\tau \to c \\ \tau \in J}} \mathfrak{w}(\tau) = 0 \qquad \text{for every endpoint } c \in \set{a,b} \text{ with } c \notin J \fstop
\end{equation}
We prove~\eqref{eq:vanishatends} for~$c=b \notin J$, the argument at~$c=a$ being identical.

Since~$\mathfrak{w}$ is locally absolutely continuous on~$J$ with~${\mathfrak{w}}'(\tau) = \tfrac12 w_\tau(X_{s \to \tau}(x))\, \mathfrak{w}(\tau)$, we have
\begin{equation}\label{eq:sqrtWH1}
\int_J |{\mathfrak{w}}'|^2(\tau) \diff \tau = \frac14 \int_J \big|w_\tau(X_{s \to \tau}(x))\big|^2 W_{s \to \tau}(x) \diff \tau \le \mcE(s,x) < \infty \fstop
\end{equation}
As~$J$ is a bounded interval, \eqref{eq:sqrtWH1} implies that~$\mathfrak{w}$ is absolutely continuous on~$J$ and admits a finite limit
\[
\mathfrak{w}(b^-) \eqdef \lim_{ \substack{\tau \uparrow b \\ \tau \in J}} \mathfrak{w}(\tau) \in [0,\infty) \fstop
\]
Assume by contradiction that~$m\eqdef \mathfrak{w}(b^-)> 0$. Then there is~$\tau_0 \in J$ such that~$W_{s \to \tau}(x) = \mathfrak{w}^2(\tau) \ge m^2/2$ for every~$\tau \in [\tau_0,b)$, whence, by the definition of the flow and by definition~\eqref{eq:defenergy} of $\mathcal{E}$,
\[
\int_{\tau_0}^{b} \Big| X_{s \to \tau}'(x) \Big|^2 \diff \tau = \int_{\tau_0}^{b} \big| v_\tau(X_{s \to \tau}(x)) \big|^2 \diff \tau \le \frac{2}{m^2} \int_{\tau_0}^b \big| v_\tau(X_{s \to \tau}(x)) \big|^2 W_{s \to \tau}(x) \diff \tau \le \frac{2\, \mcE(s,x)}{m^2} \fstop
\]
By the Cauchy--Schwarz inequality, $\tau \mapsto X_{s \to \tau}(x)$ has integrable derivative in~$(\tau_0,b)$. It is therefore absolutely continuous in~$(\tau_0,b)$ and admits a limit~$z \in \R^d$ as~$\tau \uparrow b$. Setting~$X_{s \to b}(x) \eqdef z$ we obtain an absolutely continuous map on~$[\tau_0,b]$ which satisfies the integral formulation of the characteristic system on~$[\tau_0,b]$, i.e.~the solution is defined at time~$b$ as well. This contradicts the maximality of~$J = I(s,x)$, since~$b \notin J$. Hence, it must be~$\mathfrak{w}(b^-)=0$, which proves~\eqref{eq:vanishatends} and the continuity of~$\tau \mapsto \mssT(s,x)_{\tau}$ in~$[0, T]$.

We now prove that~$\mssT$ is Borel. Let~$G \eqdef \set{ \mcE < \infty }$, which is a Borel set of full~$\rho$-measure by~\eqref{eq:energyae} and on which, by the first part of the proof,~$\mssT$ takes values in~$\cic$. We fix~$\tau \in [0,T]$ and prove that~$(s,x) \mapsto \mssT(s,x)_\tau$ is Borel on~$G$. The set~$G_\tau \eqdef \set{ (s,x) \in G : \tau \in I(s,x) }$ is Borel, being the intersection of~$G$ with a section of the open set~$\mcD$ of~\eqref{eq:defD}. On~$G \setminus G_\tau$ we have $\mssT(s, x)_{\tau} = \mfo$, while on~$G_\tau$ it is the composition of the Borel map (cf.~Lemma~\ref{le:flowreg}, Lemma~\ref{le:borelmaps}(1), and the Borel measurability of~$\theta$)
\[
(s,x) \longmapsto \big( X_{s \to \tau}(x),\, W^{1/2}_{s \to \tau}(x)/\theta(s,x) \big) \in \R^d \times [0,\infty)
\]
with the continuous projection~$(y,r) \mapsto [y,r]$ onto~$\mfC[\R^d]$. Since~$\mfC[\R^d]$ is separable, Lemma~\ref{le:curvemeas} applies and gives that~$\mssT \colon [0,T] \times \R^d \to \cic$ is Borel.

 Finally,~$\theta^2 \rho$ is a probability measure by the choice of~$\theta$.

\medskip

\noindent\paragraph{Claim~3} We have $h_t^2(\eta) = \mu_t$ for every~$t \in [0,T]$.

\smallskip

\noindent\paragraph{Proof of Claim~3} We fix~$t \in [0,T]$ and~$\phi \in \mcC_b(\R^d)$. Since two positive measures which agree on every nonnegative~$\phi \in \mcC_b(\R^d)$ coincide, we may assume~$\phi \ge 0$, so that all the integrands below are nonnegative and Tonelli's theorem applies. Using the definition of $\eta$ and $\rho$, property~\eqref{eq:Ccocycle}, Lemma~\ref{le:twotime}, and Tonelli's theorem, we get
\begin{align}
\label{e:htmu}
\int_{\R^d} \phi \diff h_t^2(\eta) &= \int_{\mcC_{T}} \phi(x_\gamma(t))r^2_\gamma(t) \diff \eta(\gamma) 
\\
&=  \int_{[0, T] \times \R^{d}} \phi(X_{s \to t}(x)) W_{s \to t}(x) \theta^{-2}(s,x) \theta^2(s,x) \car_{I(s,x)}(t) \diff \rho(s,x) \nonumber
\\
&= \int_{[0, T] \times \R^{d}} \biggl [ \phi(X_{\mfm(s,x) \to t}(X_{s \to \mfm(s,x)}(x))) W_{\mfm(s,x) \to t}(X_{s \to \mfm(s,x)}(x)) \nonumber
\\
& \quad \quad \quad C^{-1}(\mfm(s,x),X_{s \to \mfm(s,x)}(x)) \car_{I(\mfm(s,x),X_{s \to \mfm(s,x)}(x))}(t) \biggr ] \diff \Theta (s,x)\nonumber
\\
&=  \int_{[0, T] \times \R^{d}} \phi (X_{s \to t}(x)) W_{s \to t}(x) C^{-1}(s,x) \car_{I(s,x)}(t) \diff \Theta (s,x)\nonumber
\\
&= \int_0^T \int_{\R^d} \phi (X_{s \to t}(x)) W_{s \to t}(x) C^{-1}(s,X_{t \to s}(X_{s \to t}(x))) \diff \mu_s \restr{ E_{s \to t} } (x) \diff s\nonumber
\\
&= \int_0^T \int_{\R^d} \phi(x) C^{-1}(s,X_{t \to s}(x)) \diff \mu_t \restr{ E_{t \to s} }(x) \diff s\nonumber
\\
&= \int_{\R^d} \phi(x) \paren{ \int_0^T C^{-1}(s,X_{t \to s}(x))\, \car_{I(t,x)}(s) \diff s } \diff \mu_t(x)\nonumber
\\
&= \int_{\R^d} \phi(x) \paren{ \int_{I(t,x)} \frac{ W_{t \to s}(x) }{ C(t,x) } \diff s } \diff \mu_t(x) = \int_{\R^d} \phi \diff \mu_t \fstop\nonumber
\end{align}
Since~$\phi$ is arbitrary, this proves Claim~3.

\medskip

\noindent\paragraph{Claim~4} The measure $\eta$ is concentrated on curves~$\gamma \in\acic$ satisfying~\eqref{eq:ode}. \smallskip

\noindent\paragraph{Proof of Claim~4} We fix~$(s,x)$ with~$\mcE(s,x)<\infty$ and we set~$\gamma \eqdef \mssT(s,x)$; on~$I(s,x)$ we have~$x_\gamma(\tau) = X_{s \to \tau}(x)$ and~$r_\gamma(\tau) = W^{1/2}_{s \to \tau}(x)/\theta(s,x)>0$, both locally absolutely continuous, with
\[
|x_\gamma'|(\tau) = \big| v_\tau(x_\gamma(\tau)) \big| \comma \qquad  r_\gamma'(\tau) =\frac{1}{2} w_\tau(x_\gamma(\tau))\, r_\gamma(\tau) \comma \qquad k_\gamma'(\tau) = w_\tau(x_\gamma(\tau))\, k_\gamma(\tau)
\]
 for a.e.~$\tau \in I(s,x)$. Hence, by~\eqref{eq:metricderivative}, we have
\begin{equation}\label{eq:speedgamma}
|\gamma'|^2(\tau) = \theta^{-2}(s,x) \braket{ \tfrac14 \big| w_\tau(X_{s \to \tau}(x)) \big|^2 + \big| v_\tau(X_{s \to \tau}(x)) \big|^2 } W_{s \to \tau}(x) 
\end{equation}
for a.e.~$\tau \in I(s,x)$, so that~$\int_0^T|\gamma'|^2 \diff \tau = \int_{I(s,x)} |\gamma'|^2 \diff \tau = \theta^{-2}(s,x)\, \mcE(s,x) < \infty$, by definition of~$\mcE$ in~\eqref{eq:defenergy}, and the fact that~$|\gamma'|$ vanishes in~$[0,T] \setminus \overline{I(s,x)}$. By Claim~1, we deduce that for~$\rho$-a.e.~$(s,x) \in [0,T] \times \R^d$ the curve $\mssT(s,x)$ belongs to~$\acic$ and satisfies~\eqref{eq:ode}. This proves the claim and concludes the proof of Theorem~\ref{thm:supercone}.

\begin{remark}
The cross-section~$\mfm$ and the weight~$C^{-1}$ are only needed because trajectories may be born or may die inside~$[0,T]$. Assume indeed that the flow of~$v$ is globally defined, i.e.~that~$I(0,x)=[0,T]$ for~$\mu_0$-a.e.~$x \in \R^d$, and that~$m \eqdef \mu_0(\R^d) \in (0,\infty)$. Then, every trajectory is alive at time~$0$, and crosses that time exactly once: the time~$0$ is itself a cross-section of the family of trajectories, and no reference time has to be selected along the way. Accordingly, one may simply set
\[
\eta \eqdef (\mssT^0)_\sharp \paren{ m^{-1} \mu_0 } \comma \qquad \mssT^0(x)_\tau \eqdef \big[\, X_{0 \to \tau}(x) \,,\ \big( m\, W_{0 \to \tau}(x) \big)^{1/2} \,\big] \comma \quad \tau \in [0,T] \fstop
\]
Indeed~$\eta$ is a probability measure, its curves never reach the vertex, and~$E_{0 \to t} = E_{t \to 0} = \R^d$ for every~$t \in [0,T]$, so that~\eqref{eq:htdef} and~\eqref{eq:twotime} with~$s=0$ give
\[
\int_{\R^d} \phi \diff h_t^2(\eta) = \int_{\R^d} \phi\big( X_{0 \to t}(x) \big)\, W_{0 \to t}(x) \diff \mu_0(x) = \int_{\R^d} \phi \diff \mu_t \comma \qquad \phi \in \Cb(\R^d) \fstop
\]
 In the next Lemma~\ref{le:global} we show that a global bound on~$w$ together with a local one on~$v$ is enough to guarantee that the flow is globally defined.
\end{remark}

\begin{lemma}\label{le:global} Let~$v\colon [0,T] \times \R^d \to \R^d$,~$w\colon [0,T] \times \R^d \to \R$ be Borel maps, and let~$\mu \in \mcC([0,T]; \msM_+(\R^d))$ be satisfying~\eqref{eq:cercone},~\eqref{eq:pbdlocle}, and 
\begin{align} 
\label{eq:w1locconeglob}
&\int_0^T \braket{ \sup_{B} |v_t| + \Lip(v_t, B) + \sup_{\R^d}|w_t| +\Lip(w_t, B) } \diff t < \infty \comma B \Subset \R^d \fstop
\end{align} 
Then for~$\mu_0$-a.e.~$x \in \R^d$ the flow~$X$ of~$v$ is globally defined. In particular
\begin{equation}\label{eq:reprglob}
\mu_t = (X_{0 \to t})_\sharp \paren{ W_{0\to t} \, \mu_0 \restr{E_{0 \to t}} }\comma \quad t \in [0,T] \fstop
\end{equation}
\end{lemma}
\begin{proof} For every~$x \in \R^d$, set~$\tau(x)\eqdef \sup I(0,x)$. By Lemma~\ref{le:twotime}, we have
\[
(X_{0 \to t})_\sharp \paren{ W_{0\to t} \, \mu_0 \restr{E_{0 \to t}} } = \mu_t \restr{ E_{t \to 0}} \le \mu_t\comma \quad t \in [0,T] \fstop
\]
Using this inequality, we compute
\begin{align*}
\int_{\R^d} & \sup_{t \in [0,\tau(x))} |X_{0 \to t}(x)-x| \diff \mu_0(x) 
\\
& \le \int_{\R^d} \int_0^{\tau(x)} |v_t(X_{0 \to t}(x))| \diff t \diff \mu_0(x)
\\
& = \int_{\R^d} \int_0^{\tau(x)} |v_t(X_{0 \to t}(x))| e^{\int_0^t w_r(X_{0 \to r}(x)) \diff r} e^{-\int_0^t w_r(X_{0 \to r}(x)) \diff r}  \diff t \diff \mu_0(x)
\\
& \le  e^{\int_0^T \|w_r\|_\infty \diff r} \int_{\R^d} \int_0^{\tau(x)} |v_t(X_{0 \to t}(x))| W_{0 \to t}(x)   \diff t  \diff \mu_0(x)
\\
& = e^{\int_0^T \|w_r\|_\infty \diff r} \int_0^T \int_{\R^d}  |v_t(X_{0 \to t}(x))| W_{0 \to t}(x)    \diff \mu_0 \restr{E_{0 \to t}} (x) \diff t
\\
& \le e^{\int_0^T \|w_r\|_\infty \diff r} \int_0^T \int_{\R^d} |v_t|   \diff \mu_t \diff t < \infty \fstop
\end{align*}
This proves that for~$\mu_0$-a.e.~$x \in \R^d$, the map~$[0, \tau(x)) \ni t \mapsto X_{0 \to t}(x)$ is bounded; thus, by e.g.~\cite[Lemma 8.1.4]{AmbGigSav08}, it must be~$\tau(x)=T$.
\end{proof}

\begin{remark}
In view of Example~\ref{example:intro}, assumption~\eqref{eq:w1locconeglob} in Lemma~\ref{le:global} is essentially sharp. Indeed, relaxing it, it is possible to have that mass is created or destroyed during the evolution, so that uniqueness is lost, together with the representation formula~\eqref{eq:reprglob}. 
\end{remark}

\section{Proof of Theorem~\ref{thm:superconegen}}\label{s:general}

The local Lipschitz bounds~\eqref{eq:w1loccone} were used in Section~\ref{s:repr} only to construct the characteristic flow of~$v$, which is the core of the proof of Theorem~\ref{thm:supercone}. In this section we remove assumption~\eqref{eq:w1loccone}, obtaining the more general superposition principle of Theorem~\ref{thm:superconegen}. In what follows we make use of the notation for curves in the cone introduced in Section~\ref{ssec:curvcone}, and in particular the objects defined in~\eqref{eq:notcurves}, as well as the functionals defined in~\eqref{eq:action} and~\eqref{e:Gzeta}.

The proof of Theorem~\ref{thm:superconegen} follows the structure of~\cite[Theorem~8.2.1]{AmbGigSav08}. Namely, we first consider a mollification of $(v, w, \mu)$ for which the assumptions of Theorem~\ref{thm:supercone} hold and a representation is given in the form of Theorem~\ref{thm:supercone}(1)--(2). Then, we pass to the limit in the mollification parameter $\eps \to 0$, recovering a representation of~$\mu$. 

\begin{proof}[Proof of Theorem~\ref{thm:superconegen}]
We fix a mollifier~$\varrho \in \mcC^\infty_c(\R^d)$ with~$\varrho \ge 0$,~$\supp \varrho \subset B_1(0)$ and~$\int_{\R^d} \varrho \diff x=1$, we set~$\varrho_\eps(x) \eqdef \eps^{-d}\varrho(x/\eps)$ for~$\eps \in (0,1]$, and we denote by
\[
\mssg(x) \eqdef (2\pi)^{-d/2} e^{-|x|^2/2} \comma \qquad x \in \R^d \comma
\]
the standard Gaussian density; we write~$\mssg$ also for the measure~$\mssg \mcL^d \in \mcP(\R^d)$.

If~$\mu_t = 0$ for every~$t$, it is enough to take~$\eta \eqdef \delta_{\bar\gamma}$, where~$\bar\gamma$ is the constant curve equal to~$\mfo$. We may therefore assume that~$\mu$ does not vanish identically. We set
\begin{align}
\label{e:some-names}
\mcA \eqdef \int_0^T \int_{\R^d} \braket{ |v_t|^2 + |w_t|^2 } \diff\mu_t \diff t < \infty \comma \qquad \bar m \eqdef \max_{t \in [0,T]} \mu_t(\R^d) \in (0,\infty) \comma
\end{align}
and we let~$K \eqdef \int_0^T \mu_t(\R^d)\diff t \in (0,\infty)$. The Cauchy--Schwarz inequality gives
\begin{equation}\label{eq:firstpowergen}
\int_0^T \int_{\R^d} \braket{ |v_t| + |w_t| } \diff\mu_t \diff t \le\sqrt{ 2  \mcA K } < \infty \fstop
\end{equation}
We denote by~$\Theta \eqdef \int_0^T (\delta_t \otimes \mu_t)\diff t$, so that~\eqref{eq:firstpowergen} implies~$v \in L^1([0, T] \times \R^{d}, \Theta; \R^{d})$ and $w \in L^1([0, T] \times \R^{d}, \Theta)$. 

We divide the proof in 5 claims. Claim 1 is devoted to the construction of suitable regularized fields~$v^{\eps}$ and~$w^{\eps}$ and a curve of measures~$\mu^{\eps}$ approximating~$\mu$. In Claim 2 we give a representation of~$\mu^{\eps}$ in terms of a positive bounded measure~$\eta^{\eps}$ relying on Theorem~\ref{thm:supercone}. In Claim 3 we show that~$\eta^{\eps}$ is tight, and thus admits a limit~$\eta^{0}$ (up to a subsequence). In Claims 4 \& 5 we prove that $\eta^{0}$ represents $\mu$ and is concentrated on curves $\gamma \in \mathcal{AC}_{T}^{2}$ which are solutions to~\eqref{eq:odegen}.

\medskip

\noindent\paragraph{Claim~1} For~$\eps \in (0,1]$ define the finite measures and the Borel maps
\begin{align*}
\mu^\eps_t &\eqdef \mu_t * \varrho_\eps + \eps\, \mssg \comma & \mssV^\eps_t &\eqdef (v_t\mu_t)*\varrho_\eps \comma & \mssW^\eps_t &\eqdef (w_t\mu_t)*\varrho_\eps \comma
\\
v^\eps_t &\eqdef \frac{\mssV^\eps_t}{\mu^\eps_t} \comma & w^\eps_t &\eqdef \frac{\mssW^\eps_t}{\mu^\eps_t} \comma
\end{align*}
the last two quotients being quotients of densities with respect to~$\mcL^d$. Then~$v^\eps$ and~$w^\eps$ are Borel maps,~$\mu^\eps \in \mcC([0,T];\msM_+(\R^d))$, the triple~$(v^\eps, w^\eps, \mu^\eps)$ satisfies~\eqref{eq:w1loccone},~\eqref{eq:cercone} and
\begin{align}
\label{eq:unifenergy}
\int_0^T \int_{\R^d} &  \braket{ |v^\eps_t|^2 + |w^\eps_t|^2 } \diff\mu^\eps_t \diff t  \le \mcA \comma \qquad \eps \in (0,1] \fstop
\end{align}
 Moreover~$\mu^\eps_t \to \mu_t$ narrowly as~$\eps \downarrow 0$, for every~$t \in [0,T]$, the family~$\set{ \mu^\eps_t : \eps \in (0,1],\ t \in [0,T] }$ being tight with~$\mu^\eps_t(\R^d) = \mu_t(\R^d)+\eps \le \bar m + 1$.

\smallskip

\noindent\paragraph{Proof of Claim~1}
By standard properties of the convolution and  by the continuity of $\mssg$, the curve~$\mu^\eps$ belongs to~$\mcC([0,T];\msM_+(\R^d))$ and $\mu^\eps_t \to \mu_t$ narrowly as~$\eps \downarrow 0$. Since~$\supp\varrho_\eps \subset B_\eps(0) \subset B_1(0)$,
\[
\mu^\eps_t\big( \R^d \setminus B_R(0) \big) \le \mu_t\big( \R^d \setminus B_{R-1}(0) \big) + \eps\, \mssg\big( \R^d\setminus B_R(0) \big) \comma \qquad R>1 \comma
\]
and the family~$\set{\mu_t}_{t \in [0,T]}$ is tight, being a compact subset of~$\msM_+(\R^d)$ by the narrow continuity of~$\mu$ on the compact interval~$[0,T]$. The asserted tightness follows.

The densities of~$\mu^\eps_t$,~$\mssV^\eps_t$ and~$\mssW^\eps_t$ are smooth in~$x$, and~$(t,x) \mapsto \mssV^\eps_t(x)$ is Borel, being continuous in~$x$ and Borel in~$t$. The same holds for~$\mssW^\eps$, so that~$v^\eps$ and~$w^\eps$ are Borel, also noting that~$\mu_t^\eps>0$. To check~\eqref{eq:w1loccone}, let~$R>0$ and~$c_R \eqdef \min_{\overline{B_R(0)}} \mssg>0$, and set~$a_t \eqdef \int_{\R^d} \braket{|v_t|+|w_t|} \diff\mu_t$, which belongs to~$L^1([0,T])$ by~\eqref{eq:firstpowergen}. On~$B_R(0)$ we have~$\mu^\eps_t \ge \eps c_R$ and
\[
|\mssV^\eps_t| + |\mssW^\eps_t| \le \|\varrho_\eps\|_\infty\, a_t \comma \qquad |\nabla\mssV^\eps_t| + |\nabla\mssW^\eps_t| \le \|\nabla\varrho_\eps\|_\infty\, a_t \comma
\]
\[
|\nabla\mu^\eps_t| \le \|\nabla\varrho_\eps\|_\infty \bar m + \eps \|\nabla\mssg\|_\infty \comma
\]
so that, on~$B_R(0)$,
\[
|v^\eps_t| + |w^\eps_t| \le \frac{ \|\varrho_\eps\|_\infty }{ \eps c_R }\, a_t \comma
\]
\[
 |\nabla v^\eps_t| + |\nabla w^\eps_t| \le \paren{ \frac{ \|\nabla\varrho_\eps\|_\infty \bar m + \|\varrho_\eps\|_\infty \big( \|\nabla\varrho_\eps\|_\infty \bar m + \eps\|\nabla\mssg\|_\infty \big) }{ \eps^2 c_R^2 } } a_t \fstop 
\]
Since~$a \in L^1([0,T])$, assumption~\eqref{eq:w1loccone} holds for~$(v^\eps,w^\eps, \mu^{\eps})$. By linearity of the convolution and since~$\eps\mssg$ does not depend on time, \eqref{eq:cercone} holds for~$( v^\eps,w^\eps, \mu^\eps)$. Finally, since $\mu^\eps_t \ge \mu_t * \varrho_\eps$, we have
\[
\int_{\R^d} |v^\eps_t|^2 \diff\mu^\eps_t = \int_{\R^d} \frac{ |\mssV^\eps_t|^2 }{ \mu^\eps_t } \diff x \le \int_{\R^d} \frac{ |\mssV^\eps_t|^2 }{ \mu_t*\varrho_\eps } \diff x  =  \int_{\R^d} \left |\frac{(v_t \mu_t) \ast \varrho_\eps}{\mu_t \ast \varrho_\eps}\right |^2 \diff (\mu_t \ast \varrho_\eps) \le \int_{\R^d} |v_t|^2 \diff\mu_t \comma
\]
 where we have also used~\cite[Lemma 8.1.10]{AmbGigSav08} with~$p=2$,~$\mu=\mu_t$,~$E=v_t \mu_t$, and~$\varrho=\varrho_\eps$.
The same computation for~$w^\eps$ and an integration in time give~\eqref{eq:unifenergy}.

\medskip

\noindent\paragraph{Claim~2} For every~$\eps \in (0,1]$ there exists a finite measure~$\eta^\eps$ on~$\cic$ such that
\begin{equation}\label{eq:etaepsprop}
\eta^\eps\big( \cic \big) \le M^* \comma \quad h^2_t(\eta^\eps) = \mu^\eps_t \ \ \forall\, t \in [0,T] \comma \quad \int_{\mcC_{T}} \mcF \diff \eta^\eps \le \mcA \comma
\end{equation}
where~$M^* \eqdef 2(K+T)/T + 2T\mcA$, and such that~$\eta^\eps$ is concentrated on curves~$\gamma \in \acic$ which satisfy~\eqref{eq:odegen} for~$(v^\eps,w^\eps)$ and~$\Sigma(\gamma)=1$.

\smallskip

\noindent\paragraph{Proof of Claim~2} By Claim~1 we may apply Theorem~\ref{thm:supercone} to~$(v^\eps, w^\eps, \mu^\eps)$, obtaining~$\hat\eta^\eps \in \mcP(\cic)$ concentrated on curves of~$\acic$ satisfying~\eqref{eq:odegen} for~$(v^\eps,w^\eps)$, and with~$h^2_t(\hat\eta^\eps)=\mu^\eps_t$ for every~$ t \in [0,T]$.  Thus, by~\eqref{eq:odegen} for~$(v^\eps,w^\eps)$ and~\eqref{eq:unifenergy}, we get 
\begin{equation}\label{eq:energyepsbound}
\int_{\mcC_{T}} \mcF \diff\hat\eta^\eps = \int_0^T \int_{\R^d} \braket{ |v^\eps_t|^2 + \tfrac14|w^\eps_t|^2 } \diff\mu^\eps_t \diff t \le \mcA \comma
\end{equation}
while the equality~$h^2_t(\hat\eta^\eps)=\mu^\eps_t$ gives 
\begin{equation}\label{eq:l2rbound}
 \int_{\mcC_{T}}  \paren{ \int_0^T k_\gamma(t) \diff t } \diff\hat\eta^\eps(\gamma) = \int_0^T \mu^\eps_t(\R^d) \diff t = K + \eps T \le K+T \fstop
\end{equation}
Since~$r_\gamma \in \mathcal{AC} ([0,T])$ with~$|r_\gamma'| \le |\gamma'|$ for $\gamma \in \mathcal{AC}_{T}^{2} $, for every~$t \in [0,T]$ we have
\[
r^2_\gamma(t) \le 2 \paren{ \frac1T \int_0^T r^2_\gamma (s)  \diff s } + 2 \paren{ \int_0^T |r_\gamma' (s) | \diff s }^2 \le \frac{2}{T}\int_0^T k_\gamma (s)  \diff s + 2T \mcF(\gamma) \comma
\]
where we used that~$r^2_\gamma(s_0) = \frac1T\int_0^T r^2_\gamma \diff s$ for some~$s_0 \in (0,T)$. Hence, by~\eqref{eq:energyepsbound} and~\eqref{eq:l2rbound},
\begin{equation}\label{eq:sigmabound}
 \int_{\mcC_{T}} \Sigma^2 \diff\hat\eta^\eps \le \frac{2(K+T)}{T} + 2T\mcA = M^* \fstop
\end{equation}
The curves with~$\Sigma(\gamma)=0$ are constantly equal to~$\mfo$ and give no contribution to~$h^2_t (\hat\eta^\eps)$. We may therefore replace~$\hat\eta^\eps$ by its restriction to~$\set{\Sigma>0}$, which is a Borel set since~$\Sigma$ is continuous, without affecting the properties above. We finally set~$\eta^\eps \eqdef \mcR_{1/\Sigma}\hat\eta^\eps$, with the notation introduced in~\eqref{e:Rlambda}. Since~$\mfD_{1/\Sigma}$ simply multiplies the radius of each curve by a positive constant,~$\eta^\eps$ is still concentrated on curves of~$\acic$ satisfying~\eqref{eq:odegen} for~$(v^\eps,w^\eps)$ and, since~$\Sigma(\mfD_{1/\Sigma}\gamma)=1$, on curves with~$\Sigma=1$. By Lemma~\ref{le:dilation},~$h^2_t(\eta^\eps)=\mu^\eps_t$ for every~$t \in [0,T]$ and, since $\mcF$ satisfies~\eqref{e:homogeneity}, Lemma~\ref{le:dilation} together with~\eqref{eq:energyepsbound} and~\eqref{eq:sigmabound} gives
\[
 \int_{\mcC_{T}} \mcF \diff\eta^\eps = \int_{\mcC_{T}} \mcF \diff\hat\eta^\eps \le \mcA \comma \qquad \eta^\eps\big( \cic \big) =  \int_{\mcC_{T}} \Sigma^2 \diff\hat\eta^\eps \le M^* \fstop
\]

\medskip

\noindent\paragraph{Claim~3} There exist a sequence~$\eps_n \downarrow 0$ and a finite measure~$\eta^0$ on~$\cic$, concentrated on~$\set{\Sigma \le 1}$, such that~$\eta^{\eps_n} \to \eta^0$ narrowly as~$n \to \infty$.

\smallskip

\noindent\paragraph{Proof of Claim~3} By Claim~1 and \cite[Remark 5.1.5]{AmbGigSav08} applied to the tight and bounded in mass family~$\set{\mu^\eps_t : \eps \in (0,1],\ t \in [0,T]}$, there is~$\zeta \in \mcC( \R^d ; [1,\infty))$ with compact sublevels such that
\[
L \eqdef \sup_{\eps \in (0,1],\, t \in [0,T]} \int_{\R^d} \zeta \diff \mu^\eps_t < \infty \fstop
\]
Let~$\mcG_{\zeta}$ be as in~\eqref{e:Gzeta} and let $\mcK_c$ be the compact set of Lemma~\ref{le:conecompact}. We immediately see that
\[
 \int_{\mcC_{T}}  \mcG_\zeta \diff\eta^\eps = \int_0^T \int_{\R^d} \zeta \diff\mu^\eps_t \diff t \le TL \comma \qquad \eps \in (0,1] \comma
\]
while~$ \int_{\mcC_{T}} \mcF \diff\eta^\eps \le \mcA$ by~\eqref{eq:etaepsprop}. Since~$\eta^\eps$ is concentrated on~$\set{\Sigma=1}$ by Claim~2, Markov's inequality gives, for every~$c>0$,
\[
\eta^\eps\Big( \cic \setminus \mcK_c \Big) \le \frac{1}{c} \int_{\mcC_{T}} \braket{ \mcF + \mcG_\zeta } \diff\eta^\eps \le \frac{ \mcA + TL }{c} \comma
\]
 As the total masses are bounded by~$M^*$ by Claim~2, the family~$\set{\eta^\eps}_{\eps \in (0,1]}$ is tight and bounded, and Prokhorov's theorem (see e.g.~\cite[Theorem 8.6.2]{Bog07}) provides a sequence~$\eps_n \downarrow 0$ and a finite measure~$\eta^0$ with~$\eta^{\eps_n} \to \eta^0$ narrowly. Finally~$\set{\Sigma \le 1}$ is closed and contains the support of every~$\eta^\eps$. Hence, $\eta^{0}$ is concentrated on~$\set{\Sigma \le 1}$.

\medskip

\noindent\paragraph{Claim~4} We have~$h^2_t(\eta^0) = \mu_t$ for every~$t \in [0,T]$ and~$\eta^0$ is concentrated on~$\acic$ with~$ \int_{\mcC_{T}} \mcF \diff\eta^0 \le \mcA$.

\smallskip

\noindent\paragraph{Proof of Claim~4} Fix~$t \in [0,T]$ and~$\phi \in \Cb(\R^d)$. The function~$\gamma \mapsto \phi(x_\gamma(t))k_\gamma(t)$, extended by~$0$ where~$r_\gamma(t)=0$, is continuous on~$\cic$, because~$\mfC[\R^d] \ni [x,r]\mapsto \phi(x)r^2$ is continuous, and it is bounded by~$\|\phi\|_\infty$ on the set~$\set{\Sigma\le1}$, which carries every~$\eta^{\eps_n}$ and~$\eta^0$. Hence, by Claim~3 and Claim~1,
\[
\int_{\R^d} \phi \diff h^2_t(\eta^0) = \lim_n \int_{\R^d} \phi \diff h^2_t(\eta^{\eps_n}) = \lim_n \int_{\R^d} \phi \diff\mu^{\eps_n}_t = \int_{\R^d} \phi \diff\mu_t \comma
\]
which is the first assertion. The second one follows from the lower semicontinuity of~$\mcF$ (cf.~Lemma~\ref{le:conecompact}) and Fatou's Lemma. In particular~$\mcF<\infty$ for~$\eta^0$-a.e.~$\gamma$, i.e.~$\gamma \in \acic$ for~$\eta^0$-a.e.~$\gamma$.

\medskip

\noindent\paragraph{Claim~5}The measure~$\eta^0$ is concentrated on curves satisfying~\eqref{eq:odegen}.

\smallskip

\noindent\paragraph{Proof of Claim~5} For~$\phi \in \mcC^\infty_c(\R^d)$,~$\psi \in \mcC^\infty_c((0,T))$, Borel maps~$u \colon [0,T]\times\R^d \to \R^d$,~$z\colon [0,T]\times\R^d\to\R$, and~$\gamma \in \acic$, we set
\[
\mcZ^{u,z}_{\phi,\psi}(\gamma) \eqdef \int_0^T \psi'\, k_\gamma\, \phi(x_\gamma) \diff t + \int_0^T \psi\, k_\gamma \Big( \scalprod{ \nabla\phi(x_\gamma) }{ u_t(x_\gamma) } + \phi(x_\gamma)\, z_t(x_\gamma) \Big) \diff t \comma
\]
which is well defined and finite whenever
\begin{align}
    \label{e:condition}
     \int_0^T k_\gamma (t)  \braket{ |u_t (x_\gamma(t)) | + |z_t (x_\gamma(t)) | } \diff t<\infty \fstop
\end{align}
We proceed in three steps.

\emph{Step 1: the functional vanishes on solutions.} Let~$\gamma \in \acic$ satisfy~\eqref{e:condition} and
\begin{align}
    \label{e:condition2}
    x'_{\gamma}(t) = u_{t} (x_{\gamma}(t))\comma \qquad k'_{\gamma}(t) = z_{t} (x_{\gamma}(t)) k_{\gamma}(t) \quad \text{ for a.e.~} t \in O_\gamma \fstop
\end{align}
Let~$\beta_\phi(t) \eqdef k_\gamma(t)\phi(x_\gamma(t))$,~$t \in [0,T]$. Then~$\beta_\phi$ is continuous on~$[0,T]$, it vanishes on~$[0,T]\setminus O_\gamma$, and on~$O_\gamma$ it is locally absolutely continuous with
\begin{equation}\label{eq:chainrulebeta}
\beta'_\phi =  k'_\gamma \phi(x_\gamma) + k_\gamma \scalprod{ \nabla\phi(x_\gamma) }{ x'_\gamma } = k_\gamma \Big( \scalprod{\nabla\phi(x_\gamma)}{u_t(x_\gamma)} + \phi(x_\gamma) z_t(x_\gamma) \Big) \fstop
\end{equation}
The right-hand side of~\eqref{eq:chainrulebeta} belongs to~$L^1(O_\gamma)$ in view of~\eqref{e:condition}. Since~$\beta_\phi$ vanishes on the complement of~$O_\gamma$ and is continuous,~$\beta_\phi \in \mathcal{AC} ([0,T])$ with~$\beta'_\phi=0$ a.e.~outside~$O_\gamma$. Multiplying~\eqref{eq:chainrulebeta} by~$\psi$, integrating over $[0, T]$, we deduce~$\mcZ^{u,z}_{\phi,\psi}(\gamma)=0$ by an integration by parts.

\emph{Step 2: an estimate on~$\eta^0$.} Let~$u \in \mcC_c([0,T]\times\R^d;\R^d)$ and~$z \in \mcC_c([0,T]\times\R^d)$; then the functional~$\mcZ^{u,z}_{\phi,\psi}$ is continuous on~$\cic$. Indeed, its integrand is of the form~$g(t,\gamma_t)$ with~$g(t,\cdot)$ continuous on~$\mfC[\R^d]$ and~$|g(t,[x, r])| \le c(\phi,\psi,u,z)\, r^2$, so that the assertion follows by dominated convergence since $\gamma \mapsto \Sigma (\gamma) $  is  continuous. 

 We also infer that
\begin{align*}
\int_{\mcC_{T}} \int_{0}^{T} k_{\gamma} (t) [| u_{t} (x_{\gamma} (t))| + |z_{t} (x_{\gamma} (t))| ] \diff t \diff \eta^{\eps_{n}} (\gamma) = \int_{0}^{T} \int_{\R^{d}} \braket{| u_{t}| + |z_{t}|} \diff \mu^{\eps_{n}}_t \diff t <\infty \comma
\end{align*}
since $\mu^{\eps_{n}} \in \mcC([0, T]; \msM_{+} (\R^{d}))$. Hence, $\eta^{\eps_{n}}$-a.e.~$\gamma \in \cic$ satisfies the integrability~\eqref{e:condition}. Since~$\eta^{\eps_n}$ is concentrated on solutions to~\eqref{eq:odegen} for~$(v^{\eps_n},w^{\eps_n})$, for~$\eta^{\eps_n}$-a.e.~$\gamma$ we have
\[
\mcZ^{u,z}_{\phi,\psi}(\gamma) = \int_0^T \psi(t) k_\gamma(t) \Big( \scalprod{ \nabla\phi(x_\gamma) }{ (u_t - v^{\eps_n}_t)(x_\gamma) } + \phi(x_\gamma)\, (z_t-w^{\eps_n}_t)(x_\gamma) \Big) \diff t \fstop
\]
Therefore, we have
\begin{align}
\label{e:somethingZ}
 \int_{\mcC_{T}} \big| \mcZ^{u,z}_{\phi,\psi} \big| \diff\eta^{\eps_n} \le c_{\phi,\psi} \int_0^T \int_{\R^d} \braket{ |u_t - v^{\eps_n}_t| + |z_t - w^{\eps_n}_t| } \diff \mu^{\eps_n}_t \diff t \comma
\end{align}
with~$c_{\phi,\psi} \eqdef \|\psi\|_\infty \max\set{ \|\nabla\phi\|_\infty, \|\phi\|_\infty }$. We estimate the right-hand side of~\eqref{e:somethingZ}. Since~$v^\eps_t \mu^\eps_t = \mssV^\eps_t$ and~$\mu^\eps_t = \mu_t*\varrho_\eps + \eps\mssg$, recalling~\eqref{e:some-names} it holds 
\begin{align*}
\int_{\R^d} |u_t - v^\eps_t| \diff\mu^\eps_t &= \big| \mssV^\eps_t - u_t \mu^\eps_t \big|(\R^d)
\\
&\le \int_{\R^d} \bigl | \big( (v_t-u_t)\mu_t \big) * \varrho_\eps \bigr |  \diff x + \int_{\R^d} \bigl | (u_t\mu_t)*\varrho_\eps - u_t\,(\mu_t * \varrho_\eps) \bigr | \diff x+ \eps \int_{\R^d} |u_t| \mssg \diff x
\\
&\le \int_{\R^d} |v_t - u_t| \diff\mu_t + \omega_u(\eps)\, \bar m + \eps \|u\|_\infty \comma
\end{align*}
where~$\omega_u$ is the modulus of continuity of~$u$. The same estimate holds for~$z$ and~$w^\eps$, so that
\[
\limsup_n  \int_{\mcC_{T}}  \big| \mcZ^{u,z}_{\phi,\psi} \big| \diff\eta^{\eps_n} \le c_{\phi,\psi} \int_0^T \int_{\R^d} \braket{ |v_t-u_t| + |w_t-z_t| } \diff\mu_t \diff t \fstop
\]
Since~$|\mcZ^{u,z}_{\phi,\psi}|$ is continuous and non-negative, Claim~3 gives
\begin{equation}\label{eq:Zestimate}
 \int_{\mcC_{T}}  \big| \mcZ^{u,z}_{\phi,\psi} \big| \diff\eta^{0} \le c_{\phi,\psi} \int_0^T \int_{\R^d} \braket{ |v_t-u_t| + |w_t-z_t| } \diff\mu_t \diff t \fstop
\end{equation}

\emph{Step 3: conclusion.} By Claim~4 and~\eqref{eq:firstpowergen} we have that
\begin{align*}
    \int_0^T k_\gamma\braket{|v_t (x_\gamma) |+|w_t(x_\gamma) |}\diff t<\infty \qquad \text{for~$\eta^0$-a.e.~$\gamma$.}
\end{align*} 
In particular~$\mcZ^{v,w}_{\phi,\psi}$ is well defined~$\eta^0$-a.e.~in $\cic$ and
\[
\big| \mcZ^{v,w}_{\phi,\psi}(\gamma) - \mcZ^{u,z}_{\phi,\psi}(\gamma) \big| \le c_{\phi,\psi} \int_0^T k_\gamma \braket{ |v_t(x_\gamma) -u_t (x_\gamma) | + |w_t (x_\gamma) -z_t (x_\gamma) | } \diff t \fstop
\]
Integrating in~$\eta^0$ and using~\eqref{eq:Zestimate},
\begin{align}
    \label{e:Z=0}
 \int_{\mcC_{T}} \big| \mcZ^{v,w}_{\phi,\psi} \big| \diff\eta^0 \le 2c_{\phi,\psi} \int_0^T \int_{\R^d} \braket{ |v_t-u_t| + |w_t-z_t| } \diff\mu_t \diff t \fstop
\end{align}
Since~$\Theta = \int_{0}^{T} \delta_{t} \otimes \mu_{t} \diff t$ is a finite Borel measure on~$[0,T]\times\R^d$ with~$v\in L^1([0, T] \times \R^{d}, \Theta;  \R^{d})$ and $w \in L^{1} ([0, T] \times \R^d, \Theta)$ (as observed at the beginning of the proof) and the spaces~$\mcC_c([0,T]\times\R^d; \R^{d})$ and~$\mcC_c([0,T]\times\R^d)$ are dense in~$L^1([0, T] \times \R^{d}, \Theta;  \R^{d})$ and in~$L^{1} ([0, T] \times \R^d, \Theta)$, respectively, the right-hand side of~\eqref{e:Z=0} can be made arbitrarily small. Thus,~$\mcZ^{v,w}_{\phi,\psi}(\gamma)=0$ for~$\eta^0$-a.e.~$\gamma$.

Let now~$\mcS$ be the countable family of functions of the form~$\chi(\cdot-q)$ and~$\scalprod{e}{\cdot-q}\chi(\cdot-q)$, with~$q,e \in \Q^d$ and~$\chi \in \mcC^\infty_c(\R^d)$ a fixed function with~$\chi \equiv 1$ on~$B_1(0)$, and let~$\mcT \subset \mcC^\infty_c((0,T))$ be countable and dense in the uniform norm together with the first derivatives. Intersecting countably many sets of full measure, we find that for $\eta^{0}$-a.e.~$\gamma \in \acic$ it holds
\begin{align}
\label{e:finalgamma-1}
    & \int_0^T k_\gamma\braket{|v_t|+|w_t|}(x_\gamma)\diff t<\infty\,,
    \\
    & \mcZ^{v,w}_{\phi,\psi}(\gamma)=0 \quad \text{for every~$\phi \in \mcS$ and~$\psi \in \mcT$.} \label{e:finalgamma-2}
\end{align}
We fix~$\gamma \in \acic$ satisfying~\eqref{e:finalgamma-1}--\eqref{e:finalgamma-2}. For~$\phi \in \mcS$, we infer that~$\mcZ^{v,w}_{\phi,\psi}(\gamma) = 0$ for every~$\psi \in \mcC^\infty_c((0,T))$. Thus, the absolutely continuous function~$\beta_\phi = k_\gamma\phi(x_\gamma)$ has distributional derivative
\[
\beta_\phi' = k_\gamma \Big( \scalprod{\nabla\phi(x_\gamma)}{v_t(x_\gamma)} + \phi(x_\gamma) w_t(x_\gamma) \Big) \qquad \text{a.e.~in } (0,T) \fstop
\]
Comparing with the chain rule, valid a.e.~in~$O_\gamma$ by Lemma~\ref{le:accurves}, we obtain
\begin{equation}\label{eq:comparison}
k_\gamma'\, \phi(x_\gamma) + k_\gamma \scalprod{ \nabla\phi(x_\gamma) }{ x_\gamma' - v_t(x_\gamma) } = k_\gamma\, \phi(x_\gamma)\, w_t(x_\gamma) \qquad \text{a.e.~in } O_\gamma \fstop
\end{equation}
Let~$t \in O_\gamma$ be a point at which~\eqref{eq:comparison} holds for every~$\phi \in \mcS$, and let~$q \in \Q^d$ with~$|x_\gamma(t)-q|<1$. Choosing~$\phi \eqdef \chi(\cdot-q)$, which equals~$1$ with vanishing gradient near~$x_\gamma(t)$, \eqref{eq:comparison} gives~$k_\gamma'(t) = w_t(x_\gamma(t))k_\gamma(t)$. Choosing~$\phi \eqdef \scalprod{e}{\cdot-q}\chi(\cdot-q)$, whose gradient equals~$e$ near~$x_\gamma(t)$, and exploiting the previous identity we are left with~$k_\gamma(t) \scalprod{e}{x_\gamma'(t)-v_t(x_\gamma(t))} = 0$ for every~$e \in \Q^d$. Since~$k_\gamma(t)>0$, this gives~$x_\gamma'(t) = v_t(x_\gamma(t))$. Hence,~$\gamma$ satisfies~\eqref{eq:odegen}.

\paragraph{Conclusion} By Claims~4 and~5 the finite measure~$\eta^0$ is concentrated on curves of~$\acic$ satisfying~\eqref{eq:odegen} and~$h^2_t(\eta^0)=\mu_t$ for every~$t \in [0,T]$. In particular~$\eta^0 \ne 0$, because~$\mu$ does not vanish identically, so that~$M \eqdef \eta^0(\cic) \in (0,\infty)$. Applying Lemma~\ref{le:dilation} with the constant map~$\lambda \equiv M^{1/2}$ and setting $\eta \eqdef \mcR_{M^{1/2}} \eta^0$ we obtain a probability measure which is still concentrated on curves of~$\acic$ satisfying~\eqref{eq:odegen}, and which satisfies~$h^2_t(\eta) = h^2_t(\eta^0) = \mu_t$ for every~$t \in [0,T]$. The identity in~\eqref{eq:energyidgen} immediately follows.
\end{proof}

 We conclude the section with the converse of Theorem~\ref{thm:superconegen}: a measure~$\eta$ concentrated on solutions to the characteristic system in the cone produces, through the map~$h^2_t$, a solution to the continuity equation with reaction.

\begin{proposition}[Converse of Theorem~\ref{thm:superconegen}]\label{prop:converse}
Let~$v \colon [0,T]\times\R^d\to\R^d$,~$w \colon [0,T]\times\R^d\to\R$ be Borel maps and let~$\eta\in\msM_+(\cic)$ be concentrated on curves~$\gamma\in\acic$ satisfying~\eqref{eq:odegen}. Set~$\mu_t\eqdef h^2_t(\eta)$ for~$t\in[0,T]$. If
\begin{equation}\label{eq:pbd1converse}
\int_0^T\int_{\R^d}\braket{1+|v_t|+|w_t|}\diff\mu_t\diff t<\infty\comma
\end{equation}
then~$\mu\in\mcC([0,T];\msM_+(\R^d))$ and~$\mu$ solves~\eqref{eq:cerconegen}.
\end{proposition}
\begin{proof}
\emph{Finiteness and continuity of~$\mu$.} By~$\mu_t\eqdef h^2_t(\eta)$ for every~$t \in [0,T]$, together with Tonelli's theorem,
\[
\int_0^T\mu_t(\R^d)\diff t=\int_{\cic}\int_0^T k_\gamma(t)\diff t\diff\eta(\gamma) \comma\qquad \int_0^T\int_{\R^d}|w_t|\diff\mu_t\diff t=\int_{\cic}\int_0^T k_\gamma(t)|w_t|(x_\gamma(t))\diff t\diff\eta(\gamma)\comma
\]
both finite by~\eqref{eq:pbd1converse}. In particular, for~$\eta$-a.e.~$\gamma$ it holds~$\int_0^T k_\gamma(t)|w_t|(x_\gamma(t))\diff t<\infty$; we fix such a~$\gamma$. Since~$\gamma\in\acic$ satisfies~\eqref{eq:odegen}, the function~$k_\gamma$ is absolutely continuous on~$[0,T]$ by Lemma~\ref{le:accurves}, with~$k_\gamma'(t)=w_t(x_\gamma(t))k_\gamma(t)$ for a.e.~$t \in O_\gamma$ and~$k_\gamma'=0$ a.e.~on~$[0,T]\setminus O_\gamma$. Picking~$s_0\in(0,T)$ with~$k_\gamma(s_0)=\tfrac1T\int_0^T k_\gamma(s)\diff s$, we get, for every~$t\in[0,T]$,
\[
k_\gamma(t)\le k_\gamma(s_0)+\int_0^T|k_\gamma'|(s)\diff s\le\frac1T\int_0^T k_\gamma(s)\diff s+\int_0^T k_\gamma(s)|w_s|(x_\gamma(s))\diff s\comma
\]
whence, taking the maximum in~$t\in[0,T]$ and integrating w.r.t.~$\eta$,
\begin{equation}\label{eq:sigmaconverse}
\int_{\cic}\Sigma(\gamma)^2\diff\eta(\gamma)\le\frac1T\int_0^T\mu_t(\R^d)\diff t+\int_0^T\int_{\R^d}|w_t|\diff\mu_t\diff t<\infty\fstop
\end{equation}
Consequently~$\mu_t(\R^d)=\int_{\cic}k_\gamma(t)\diff\eta(\gamma) \le\int_{\cic}\Sigma^2\diff\eta<\infty$ for every~$t\in[0,T]$, so that~$\mu_t\in\msM_+(\R^d)$; moreover, for~$\phi\in\Cb(\R^d)$, the map~$[0,T] \ni t\mapsto\int_{\R^d}\phi\diff\mu_t=\int_{\cic}\phi(x_\gamma(t))k_\gamma(t)\diff\eta$ is continuous by dominated convergence, its integrand being continuous in~$t$ and dominated by~$\|\phi\|_\infty\Sigma^2\in L^1(\cic, \eta)$. Hence~$\mu\in\mcC([0,T];\msM_+(\R^d))$.

\emph{The equation.} Fix~$\zeta\in\mcC^\infty_c((0,T)\times\R^d)$ and set~$C_\zeta\eqdef\sup\big(|\partial_t\zeta|+|\nabla\zeta|+|\zeta|\big)$. Using Tonelli's theorem,
\begin{equation}\label{eq:fubiniconverse}
\int_{\cic}\int_0^T k_\gamma(t)\braket{|\partial_t\zeta_t|+|\nabla\zeta_t|\,|v_t|+|\zeta_t|\,|w_t|}(x_\gamma(t))\diff t\diff\eta(\gamma) \le C_\zeta\int_0^T\int_{\R^d}\braket{1+|v_t|+|w_t|}\diff\mu_t\diff t\comma
\end{equation}
which is finite by~\eqref{eq:pbd1converse}. In particular the inner integral is finite for~$\eta$-a.e.~$\gamma$.

Let~$\gamma\in\acic$ satisfy~\eqref{eq:odegen} and~$\int_0^T k_\gamma(t)\braket{|v_t|+|w_t|}(x_\gamma(t))\diff t<\infty$, which by~\eqref{eq:fubiniconverse} holds for~$\eta$-a.e.~$\gamma$. Set~$\beta_\zeta(t)\eqdef\zeta_t(x_\gamma(t))k_\gamma(t)$,~$t \in [0,T]$. Exactly as for~$\beta_\phi$ in Step~1 of Claim~5 in the proof of~Theorem~\ref{thm:superconegen},~$\beta_\zeta$ is continuous on~$[0,T]$, vanishes on~$[0,T]\setminus O_\gamma$, and on~$O_\gamma$ it is locally absolutely continuous with, by~\eqref{eq:odegen},
\[
\beta_\zeta'(t)=k_\gamma(t)\braket{\partial_t\zeta_t(x_\gamma(t))+\scalprod{\nabla\zeta_t(x_\gamma(t))}{v_t(x_\gamma(t))}+\zeta_t(x_\gamma(t))\,w_t(x_\gamma(t))} \quad \text{ for a.e.~} t \in O_\gamma \fstop
\]
Its right-hand side belongs to~$L^1(O_\gamma)$ by the choice of~$\gamma$; since~$\beta_\zeta$ is continuous and vanishes off~$O_\gamma$, we get~$\beta_\zeta\in \mathcal{AC}([0,T])$ with~$\beta_\zeta'=0$ a.e.~outside~$O_\gamma$. As~$\zeta$ has compact support in~$(0,T)\times\R^d$, $\beta_\zeta(0)=\beta_\zeta(T)=0$, so that
\[
0=\int_0^T\beta_\zeta'(t)\diff t=\int_0^T k_\gamma(t)\braket{\partial_t\zeta_t(x_\gamma(t))+\scalprod{\nabla\zeta_t(x_\gamma(t))}{v_t(x_\gamma(t))}+\zeta_t(x_\gamma(t))\,w_t(x_\gamma(t))}\diff t\fstop
\]
Integrating w.r.t.~$\eta$ (possible by~\eqref{eq:fubiniconverse}) we obtain
\[
0=\int_0^T\int_{\R^d}\braket{\partial_t\zeta_t+\scalprod{\nabla\zeta_t}{v_t}+\zeta_t\,w_t}\diff\mu_t\diff t\fstop
\]
Since~$\zeta\in\mcC^\infty_c((0,T)\times\R^d)$ is arbitrary, this is precisely~\eqref{eq:cerconegen}.
\end{proof}

\section*{Acknowledgments}

 The work of SA was supported by the FRA2022 Project ``ReSinAPAS'' and by the Austrian Science Fund (FWF) through the project \href{https://www.fwf.ac.at/en/research-radar/10.55776/P35359}{10.55776/P35359}. SA is also member of the Gruppo Nazionale per l'Analisi Matematica, la Probabilit\`a e le loro Applicazioni (INdAM-GNAMPA) and acknowledges the support of the INdAM-GNAMPA 2026 Project ``Sistemi multi-agente e replicatore: derivazione particellare e ottimizzazione'' CUP E53C25002010001. The research of GES was supported by the Austrian Science Fund (FWF) through the project \href{https://www.doi.org/10.55776/F100800}{10.55776/F100800}.
For open-access purposes, the authors have applied a CC BY public copyright
license to any author-accepted manuscript version arising from this
submission.
\medskip

\textbf{Declaration of Generative AI in Scientific Writing. } During the preparation of this work, the authors used AI tools to assist in drafting initial text for routine technical descriptions and to perform language editing and stylistic revisions. All AI-generated content was thoroughly reviewed, verified, and refined by the authors, who retain full responsibility for the scientific accuracy, integrity, and final content of the manuscript.

\end{document}